\documentclass[12pt]{article}
\usepackage[a4paper,margin=2.5cm]{geometry}
\usepackage{showlabels}
\usepackage{setspace}

\usepackage{amsfonts,graphicx,amsmath,amssymb,amsthm,color,nicefrac,enumerate, layout, bbm, listings,  hyperref}
\usepackage[latin1]{inputenc}
\usepackage{natbib,expdlist}
\usepackage{amsmath}
\usepackage{amsfonts}
\usepackage{amssymb}
\numberwithin{equation}{section}

\newtheorem{lemma}{Lemma}
\newtheorem{theorem}{Theorem}

\theoremstyle{definition}

\theoremstyle{remark}

\newtheorem{remark}{Remark}

\newcommand{\convas}[1][]{\xrightarrow[#1]{\mathrm{a.s.}}}
\newcommand{\convp}[1][]{\xrightarrow[#1]{\PP}}
\newcommand{\convl}[1][]{\xrightarrow[#1]{\mathrm{law}}}
\newcommand{\convsl}[1][]{\xrightarrow[#1]{\mathrm{stably}}}

\newcommand{\NN}{\mathbb{N}}
\newcommand{\RR}{\mathbb{R}}
\newcommand{\PP}{\mathbb{P}}
\newcommand{\EE}{\mathbb{E}}

\newcommand{\cG}{\mathcal{G}}

\newcommand{\cM}{\mathfrak{M}}

\newcommand{\cQ}{\mathfrak{Q}}

\newcommand{\ind}[1]{\mathbf{1}_{#1}}
\newcommand{\as}{\mathrm{a.s.}}

\newcommand{\vd}{\,\mathrm{d}}

\newcommand{\dvec}{ (a, b) }
\newcommand{\dvecj}{ (a_j, b_j) }
\newcommand{\destimj}{ ( \alpha_T^{(j)},\beta_T^{(j)}) }
\newcommand{\destim}{ ( \alpha_T,\beta_T) }
\newcommand{\destimNj}[1]{ ( \widehat a_{#1,N}^{(j)}, \widehat b_{#1,N}^{(j)} )}
\newcommand{\destimN}[1]{ ( \widehat a_{#1,N}, \widehat b_{#1,N})}

\DeclareMathOperator{\sgn}{sgn}

\newcommand{\eqdef}{\mathbin{:=}}

{\begin{itemize}[leftmargin=1cm,noitemsep,label={$\star$}]}{\end{itemize}}

\begin{document}

\title{
Drift estimation of the threshold Ornstein-Uhlenbeck process from continuous and discrete observations
}
\author{
Sara Mazzonetto\footnote{Universit\'e de Lorraine, CNRS, Inria, IECL, F-54000 Nancy, France. E-mail: \texttt{sara.mazzonetto@univ-lorraine.fr}}
\ and
Paolo Pigato\footnote{%
Department of Economics and Finance, University of Rome Tor Vergata, Via Columbia 2, 00133 Roma, Italy. 
E-mail: \texttt{paolo.pigato@uniroma2.it} 
{\smallskip\newline}
{\bf Acknowledgements.} We are grateful to G. Conforti, A. Lejay and E. Mariucci for discussion, and to the referees of Statistica Sinica for their helpful remarks and suggestions.}}

\date{\today}

\maketitle

\begin{abstract}
\noindent
We refer by \emph{threshold Ornstein-Uhlenbeck} to a continuous-time threshold autoregressive process. It follows the Ornstein-Uhlenbeck dynamics when above or below a fixed level, yet at this level (threshold) its coefficients can be discontinuous. We discuss (quasi)-maximum likelihood estimation of the drift parameters, both assuming continuous and discrete time observations. In the ergodic case, we derive consistency and speed of convergence of these estimators in long time and high frequency. Based on these results, we develop a test for the presence of a threshold in the dynamics. Finally, we apply these statistical tools to short-term US interest rates modeling.
\end{abstract}

\bigskip


\bigskip

\noindent{\textbf{AMS 2010: }} primary: 62M05; secondary: 62F12; 60J60.

\section{Introduction}
\label{sec:intro}

We consider the diffusion process solution to the following stochastic differential equation (SDE)
\begin{equation}
    \label{eq:AffDOBM}
    X_t=X_0+\int_0^t \sigma(X_s)\vd W_s+\int_0^t \left( b(X_s) - a(X_s) \, X_s \right) \vd s , \quad t\geq 0,
\end{equation}
with piecewise constant volatility coefficient, possibly discontinuous at $r\in \RR$,
\begin{equation}
    \label{sigmaDOBM}
    \sigma(x)= \sigma_+ \ind{\{x\geq r\}} + \sigma_- \ind{\{x < r\}} >0
\end{equation}
and similarly piecewise affine drift coefficient
\begin{equation}\label{SETvas}
    b(x)= b_+ \ind{\{x\geq r\}} + b_- \ind{\{x < r\}}
		\quad\text{and}\quad
    a(x)= a_+ \ind{\{x\geq r\}} + a_- \ind{\{x < r\}}.
\end{equation}
Strong existence of a {unique} solution to \eqref{eq:AffDOBM} follows
from the results of \cite{legall}.  Separately on $(r,\infty)$ and $(-\infty,r)$, the process follows the Ornstein-Uhlenbeck (OU) dynamics which, in the context of interest rates modeling, is referred to as \emph{Vasicek model}. Following this nomenclature, \cite{interestrate} refer to \eqref{eq:AffDOBM} as \emph{Self Exciting Threshold (SET) Vasicek model}. \cite{su2015,su2017} refer to such model as \emph{threshold diffusion} (TD) or first-order continuous-time threshold autoregressive model (see also \citep{tong1990}).

The process is~\emph{ergodic} as long as the drift pushes the process up when the process reaches large, negative values, and down when it reaches large, positive values. Note that if $a_- > 0$, the drift points towards $b_-/a_-$ when $X_s$ is below the threshold $r$, if $a_+ > 0$, towards $b_+/a_+$ when $X_s$ is above $r$. We allow here a null linear part, and if $a_-=0$ we need $b_->0$ to push the process up when very negative and if $a_+=0$ we need $b_+<0$ to push it down when very positive. From these considerations, ergodicity can be easily checked (see explicit condition~\eqref{eq:ergodic:cases:th} below). 

We also consider, in Section~\ref{sec:multit}, a \emph{multi-threshold} version of \eqref{eq:AffDOBM}, where we allow for $d$ discontinuity levels $r_{1}<\dots<r_{d}$. In this case, ergodicity is determined by the same conditions, to be checked on the values of the coefficients on the intervals $(-\infty,r_{1})$ and $[r_{d},+\infty)$.

In this paper {we discuss the asymptotic behavior} of maximum likelihood estimators (MLE) and quasi-maximum likelihood estimators (QMLE) for the drift parameters $(a_-,a_+,b_-,b_+)$, both from continuous and discrete time observations. 
Let $N$ be the number of observations, $T_N$ the time horizon and $\Delta_{N}$ the largest interval between two consecutive observations.
In the ergodic case, if $T_N\to \infty$ and $T_N \Delta_{N} \to 0$ as $N\to \infty$, we prove a central limit theorem (CLT) giving the convergence of the estimators with speed $\sqrt{T_N}$ to the real parameters, {i.e.~asymptotic normality} (see~Theorem~\ref{th:joint:CLT} below). To the best of our knowledge, this is the first result of this kind for TDs (SDEs with discontinuous {-drift and diffusion-} coefficients).
\\
The discontinuity in the coefficients makes it difficult to pass from discrete to continuous time observations.
Indeed, a precise analysis of the error hinges on the behavior of certain discretizations of the local time of the diffusion at the threshold.
{We also 
prove, for fixed time horizon, 
that the discrete (Q)MLE based on $N$ equally spaced observations converges in high frequency to the continuous (Q)MLE, with speed $N^{1/4}$ (see Theorem~\ref{th:disc} below). 
This slow convergence of the discrete (Q)MLE to the continuous (Q)MLE follows from the slow convergence, with speed $N^{1/4}$, of the discretization of the local time. }
\\
Based on these results we provide a test to decide whether a threshold is present in the dynamics. Finally we use these tools to analyze short term US interest rates.

\paragraph{Literature review.}
\cite{su2015,su2017}  study the asymptotic behavior of the continuous time QMLE of a TD with drift as in \eqref{SETvas} and piecewise regular diffusivity. 
In particular, they construct a hypothesis test to decide
whether the drift is affine or piecewise affine.

The estimation of the volatility parameters $\sigma_\pm$ in~\eqref{eq:AffDOBM}
from high-frequency data is studied in \citep{LP} and the drift estimation in case $a_\pm=0$ in \citep{lp2}. In the purely linear drift case $b_\pm=0$, \cite{kutoyants2012} studies the problem of identifying the threshold parameter $r$ and  \cite{dieker2013} and the related stream of research consider similar models, with $r=0$ (so that the drift function is continuous) in a multidimensional setting. {The (related) problem of drift estimation in a skew OU process is considered in \citep{XingZhaoLi2020}.}

In this document the coefficients are discontinuous and the behavior at $r$ hard to handle; for high-frequency observations we do so using the discretization results in \citep{mazzonetto2019rates}. 
The convergence in high frequency and long time for estimators of discretely observed diffusions have been discussed e.g. in \citep{kessler,alaya_kebaier_2013,amorino_glotier}, but to the best of our knowledge ours is the first such result in the case of discontinuous coefficients. 

\cite{YU2020}
study {numerically} an approximate 
MLE (AMLE)
from discrete time observations 
simultaneously for threshold, drift and diffusion coefficients of threshold diffusions including OU process or CIR model. They compare their AMLE with QMLE, showing numerical evidence for consistency. 
The recent work \citep{HuXi} considers a generalized moment estimator for a TD which is discretely observed, with fixed time lag.

Threshold autoregressive (TAR) models in discrete time were introduced by  H. Tong in the early 1980s \citep{tong1983,Tong:2011ud,tong2015}. Within this class, self-exciting TAR (SETAR) models rely on a spatial segmentation, with a change in the dynamics according to the position of the process,
below or above a threshold, and can be seen as a discrete analogue to the TD. We refer to \citep{chan1993,Rabemananjara:1993dh,Yadav:1994km,brockwell_williams_1997,Chen:2011bk} and references therein for this class of econometric models and related inference problems. 

Diffusion processes have been widely used to model interest rate dynamics, celebrated classical examples being \citep{vasicek,cir,Hull_White,Black_Karasinski}. These models are designed to capture the fact that interest rates are typically mean reverting, see
\citep{WuZhang}.  However, non linear effects (e.g.~multi-modality) are not captured by these models.  \cite{aitsahalia} shows that mean-reversion for interest rates is strong outside a middle region, suggesting the existence of a target band. This is similar to what is observed in exchange rates  \citep{krugman} and explainable with policy adjustments in response to changes in such rates. There is evidence for a ``normal'' low-mean regime and an ``exceptional'' high-mean regime, and in general for bi-modality (or even multi-modality) in interest rate dynamics, that one can model using TD \eqref{SETvas}.  In general, non-linearities and regime changes in short-term interest rates have been widely documented, and several threshold models have been proposed both in discrete and continuous time, see
\citep{GRAY199627,pfann,AngBekaert,AngBekaert2,Kalimipalli,Gospodinov,AngBekaert3,LemkeArchontakis,LemkeArchontakis2}. 
We refer to \citep{interestrate} and bibliography therein for a thorough discussion of SET diffusions in interest rate modeling.
In recent years TDs have been used in several aspects of financial modeling, such as  option pricing \citep{liptonsepp,gairat,DongWong,lipton:2018,pigato} and  time series modeling \citep{ang,lp1}. TD models for interest rates have been considered in \citep{pai,interestrate,su2015,su2017}.
In this paper, we focus on (Q)MLE estimation of such models, and in particular on inference from high frequency observations and their convergence to continuous time estimators, as well as their convergence in long time to real values of the parameters.

\paragraph{Outline.} {In Section \ref{sec:mainresults} we present our main results on convergence of drift estimators for threshold OU. 
In Section \ref{sec:numerics} we  implement the estimators, discuss threshold estimation and testing and work with US interest rates data.
Proofs are collected in Section \ref{sec:proofs} and in Section~\ref{sec:multit} we discuss an extension to a multi-threshold setting}.

\section{(Quasi) maximum likelihood estimation}\label{sec:mainresults}
Let $X$ be the process strong solution to~\eqref{eq:AffDOBM} where $W$ is a Brownian motion and $X_0$ is independent of $W$ (e.g., $X_0$ is deterministic).



We see in next equation~\eqref{eq:ergodic:cases} that $X$ is ergodic if 
\begin{equation} \label{eq:ergodic:cases:th}
\begin{split}
& \textrm{
[($a_+>0$ and $b_+\in \RR$) or ($a_+=0$ and  $b_+< 0$)]} 
\\
& \textrm{and [($a_->0$ and $b_-\in \RR)$ or ($a_-=0$ and  $b_-> 0$)].
}
\end{split}
\end{equation}

\subsection{Maximum and quasi-maximum  likelihood  estimator from continuous time observations}
We assume in this section to observe the process on the time interval $[0,T]$, $T\in (0,\infty)$.
For $T\in (0,\infty)$ and $m=0,1,2$, we define
\begin{equation}\label{def:M:Q}
 \cM_T^{\pm,m}:=
	 \int_0^T  X_s^m \ind{\{ \pm (X_s-r) \geq 0 \} } \vd X_s
\quad \text{and} \quad	
\cQ_T^{\pm,m} 	
:=
\int_0^T X_s^m \ind{\{ \pm (X_s-r) \geq 0 \}}\vd s
\end{equation}
and take as likelihood function
the Girsanov weight 
\begin{align} \label{eq:RN}
	G_T(a_+, b_+, a_-, b_- ) 
	= 
	\exp{\left(\int_0^T \frac{b(X_s) - a(X_s)X_s}{(\sigma(X_s))^2} \vd X_s - \frac12 \int_0^T \frac{(b(X_s) - a(X_s)X_s)^2}{(\sigma(X_s))^2} \vd s \right)}.
\end{align}
We also consider a quasi-likelihood defined as in \citep{su2015} as
\begin{equation} \label{QMLE:eq0}
	\Lambda_T (a_+, b_+, a_-, b_- ) 
	= \int_0^T b(X_s) - a(X_s)X_s  \vd X_s  - \frac12 \int_0^T (b(X_s) - a(X_s)X_s)^2 \vd s.
\end{equation}

\begin{theorem} \label{th:continuous}	
Let $\pm\in \{+,-\}$. 
\begin{enumerate}[i)]
\item \label{th1:item:QMLE}
For every $T\in (0,\infty)$ the MLE and QMLE are given by
\begin{equation} \label{eq:QMLE_ct}	
\begin{pmatrix}
	\alpha^\pm_T, 
	&
	\beta^\pm_T 
\end{pmatrix}
=
\begin{pmatrix}
	\frac{	\cM^{\pm,0}_T  \cQ^{\pm,1}_T			
	-\cQ^{\pm,0}_T
\cM^{\pm,1}_T
	}{ 
\cQ^{\pm,0}_T 
\cQ^{\pm,2}_T	-(\cQ^{\pm,1}_T)^2},
		&
	\frac{
\cM^{\pm,0}_T \cQ^{\pm,2}_T-\cQ^{\pm,1}_T	
\cM^{\pm,1}_T	}{
\cQ^{\pm,0}_T
\cQ^{\pm,2}_T-(\cQ^{\pm,1}_T)^2
}
	\end{pmatrix}.
\end{equation}
\end{enumerate}
Assume now that the process is ergodic.
\begin{enumerate}[i)]
\setcounter{enumi}{1}
\item \label{th1:item:LLN}
The following law of large numbers (LLN) holds:
$
(\alpha^\pm_T-a_\pm , \ \beta^\pm_T-b_\pm) \convas[T\to\infty] 0,
$
i.e.,~the estimator is consistent.
\item \label{th1:item:CLT}
The following CLT holds: 
$\sqrt{T}
\begin{pmatrix}
\alpha^\pm_T-a_\pm ,
& \beta^\pm_T-b_\pm
\end{pmatrix}
\xrightarrow[{T\to \infty}]{\mathrm{stably}}  
N^\pm=\begin{pmatrix}
N^{\pm, \alpha}, & N^{\pm,\beta}
\end{pmatrix}
$ 
where  
$\begin{pmatrix}
	N^{+,\alpha}, & N^{+,\beta}
\end{pmatrix}
$
and 
$\begin{pmatrix}
	N^{-,\alpha}, & N^{-,\beta}
\end{pmatrix}
$
are two independent, independent of $X$, two-dimensional Gaussian random variables 
{with covariance matrices respectively $\sigma_+^2 \Gamma_+^{-1}$ and $\sigma_-^2 \Gamma_-^{-1}$ 
where
\begin{equation} \label{th:erg:covariance}
\Gamma_\pm :=  
\begin{pmatrix} \overline{\cQ}^{\pm,2}_\infty &  - \overline{\cQ}^{\pm,1}_\infty \\
 - \overline{\cQ}^{\pm,1}_\infty & \overline{\cQ}^{\pm,0}_\infty 
\end{pmatrix}, 
\end{equation}
and $\overline{\cQ}^{\pm,i}_\infty ,\, i\in \{0,1,2\}$
are real constants such that
$\lim_{t\rightarrow \infty} \frac{{\cQ}^{\pm,i}_t}{t} \overset{\as}= \overline{\cQ}^{\pm,i}_\infty$ 
(explicit expressions for such constants are given in next Lemma~\ref{eq:Qinfinity2}; {more details on 
stable convergence 
can be found 
in
Remark~\ref{rem:stable} below}).}
\item \label{th1:item:LAN}
The \emph{local asymptotic normality} (LAN) property (see~\cite{lecam00}) holds for the likelihood evaluated at the true parameter
$\theta := (a_+,b_+,a_-,b_-)$
with rate of convergence
$\frac1{\sqrt{T}}$ and asymptotic Fisher information 
\[ \Gamma 
	= \begin{pmatrix} \sigma_+^{-2} \Gamma_+ & 0_{\RR^{2\times 2}} \\
		0_{\RR^{2\times 2}} & \sigma_-^{-2} \Gamma_- \end{pmatrix}.\]
This means that there exists a random vector 
${A_T}
\in \RR^4$ 
such that for all small perturbations $\Delta \theta :=(\Delta a_+,\Delta b_+, \Delta a_-,\Delta b_-)$ it holds that the quantity
\begin{equation} \label{eq:LAN}
\begin{split}
	& {\scriptsize \log \frac{ G_T(\theta + \frac{1}{\sqrt{T}}\Delta \theta)} 
	{G_T(\theta)}}
	{\scriptsize 
	- \left(  \Delta \theta \cdot A_T - \frac1{2} \Delta \theta \cdot  \Gamma 
\Delta \theta \right)}
\end{split}
\end{equation}
converges to 0 in probability as $T\to\infty$.
\end{enumerate}
\end{theorem}

\begin{remark} \label{rem:stable}
The notion of \emph{stable convergence} was introduced by \cite{renyi}. We refer to \citep{js} or \citep{jp} for a detailed exposition. 
%
In this document we just mention that this notion of convergence is stronger than convergence in law, but weaker than convergence in probability.  We use in this paper the following crucial property: for random variables $Y_n$, $Z_n$ ($n\geq 1$), $Y$ and $Z$, 
\begin{equation}\label{scl}
    \text{if }
Z_n \convsl[n\to\infty] Z
\text{ and }
Y_n \convp[n\to\infty] Y
\text{ then }
(Y_n,Z_n) \convsl[n\to\infty] (Y,Z).
\end{equation}
\end{remark}

\subsection{Drift estimation from discrete observations}
\label{sec:est:disc}

We assume in this section to observe the process on the discrete time grid $0=t_0 < t_1 <\ldots<t_{N-1}< t_N=T$, 
for $N\in \NN$, $T\in (0,\infty)$, and set $\Delta_N=\max_{k=1,\dots, N} \{t_{k}-t_{k-1}\}$. 
We define  $X_i\eqdef X_{t_i}$  with $i=0,\ldots,N$.

The discrete versions of \eqref{def:M:Q} are defined as follows: for $m=0,1,2$ and $\pm\in \{-,+\}$, let
\begin{equation} \label{def:M:Q:disc}
\begin{split}
   & \cM^{\pm,m}_{T,N} := \sum_{k=0}^{N-1} 
    X_k^m \ind{ \{ \pm (X_k-r) \geq 0 \} } (X_{k+1}-X_k),
\quad \text{and} \\
&\quad
    \cQ^{\pm,m}_{T,N} := \sum_{k=0}^{N-1} 
    X_k^m \ind{\{ \pm (X_k-r) \geq 0 \}} (t_{k+1}-t_k).
\end{split}
\end{equation}
We refer with \emph{discretized likelihood} (corresponding to \eqref{eq:RN}) 
to
\begin{equation*}
\begin{split}
  &  G_{T,N}(a_+,b_+, a_-,b_-)\\
  & =\exp\bigg(\sum_{i=0}^{N-1} \bigg(\frac{b(X_i) - a(X_i) X_i }{\sigma(X_i)^2}(X_{i+1}-X_i)-  \frac{t_{i+1}-t_i}{2}\frac{(b(X_i) - a(X_i) X_i )^2}{\sigma(X_i)^2}  \bigg)\bigg)
  \end{split}
\end{equation*}
and with \emph{discretized quasi-likelihood}
(corresponding to \eqref{QMLE:eq0}) to 
\begin{equation*}
\begin{split}
    &\Lambda_{T,N}(a_+,b_+, a_-,b_-)\\
    &=\sum_{i=0}^{N-1} \left( (b(X_i) - a(X_i) X_i ) (X_{i+1}-X_i)-
    \frac{t_{i+1}-t_i}{2} (b(X_i) - a(X_i) X_i )^2
    \right).
\end{split}
\end{equation*}
For $N \in \NN$, $T\in (0,\infty)$, let
\begin{equation} \label{eq:QMLE_dis_time}	
\begin{pmatrix}
 \widehat{a}^\pm_{T,N},
		&
 \widehat{b}^\pm_{T,N}
	\end{pmatrix}
=
\begin{pmatrix}
 	 \frac{\cM^{\pm,0}_{T,N} \cQ^{\pm,1}_{T,N} - \cQ^{\pm,0}_{T,N}
\cM^{\pm,1}_{T,N}	
	}{ 
\cQ^{\pm,0}_{T,N}
\cQ^{\pm,2}_{T,N}	-(\cQ^{\pm,1}_{T,N})^2},
		&
 	\frac{
\cM^{\pm,0}_{T,N} \cQ^{\pm,2}_{T,N}-\cQ^{\pm,1}_{T,N}	
\cM^{\pm,1}_{T,N}	}{
\cQ^{\pm,0}_{T,N}
\cQ^{\pm,2}_{T,N}-(\cQ^{\pm,1}_{T,N})^2
}
	\end{pmatrix}.
\end{equation}

\begin{theorem} \label{th:joint:CLT}
Let $(T_N)_{N\in \NN}$ be a sequence in $(0,\infty)$.
For all $N\in \NN$, let $\Delta_N$ above be $T_N/N$, let $ \widehat{a}^\pm_{T_N,N},\widehat{b}^\pm_{T_N,N}$ be defined as in \eqref{eq:QMLE_dis_time}.
\begin{enumerate}[i)]
\item \label{th2:item:QMLE}
For every $N\in \NN$ the vector $(\widehat{a}^+_{T_N,N},\widehat{b}^+_{T_N,N}, \widehat{a}^-_{T_N,N}, \widehat{b}^-_{T_N,N})$ maximizes both the likelihood 
$G_{T_N,N}(a_+,b_+, a_-b_-)$ and the quasi-likelihood $\Lambda_{T_N,N}(a_+,b_+, a_-b_-)$.
\end{enumerate}
Assume now that the process is ergodic, i.e., \eqref{eq:ergodic:cases:th} is satisfied, and that $X$ is the stationary solution to \eqref{eq:AffDOBM}, i.e., $X_0$ follows {the stationary distribution} (cf.~\eqref{eq:stat:dist}). Moreover, assume
\[
\lim_{N\to\infty}T_N = \infty \quad \text{and} \quad \lim_{N\to\infty} \Delta_N = 0.
\]
\begin{enumerate}[i)]
\setcounter{enumi}{1}
\item \label{th2:item:LLN}
The following LLN holds:
$
(
    \widehat{a}^\pm_{T_N,N}, \ \widehat{b}^\pm_{T_N,N}
  )
\convp[N\to\infty]
   ( a_\pm, \ b_\pm )
$ 
(the estimator is consistent). 
\item \label{th2:item:CLT}
If $\lim_{N\to\infty} T_N \Delta_N = 0$, the following CLT jointly holds for the positive and negative sides: 
\[
\sqrt{T_N}
\begin{pmatrix}
 \widehat{a}^\pm_{T_N,N}
-a_\pm , 
&
 \widehat{b}^\pm_{T_N,N}
-b_\pm
\end{pmatrix}
\xrightarrow[N\to \infty]{\mathrm{stably}}
N^\pm=\begin{pmatrix}
N^{\pm, \alpha}, & N^{\pm,\beta}
\end{pmatrix}
\]
where 
$\begin{pmatrix}
	N^{+,\alpha}, & N^{+,\beta}
\end{pmatrix}
$
and 
$\begin{pmatrix}
	N^{-,\alpha}, & N^{-,\beta}
\end{pmatrix}
$
are as in Theorem~\ref{th:continuous}.
\item \label{th2:item:LAN}
If $\lim_{N\to\infty} T_N \Delta_N = 0$, the discretized likelihood satisfies~\eqref{eq:LAN} (with $T=T_N$).
\end{enumerate}
\end{theorem}

\begin{remark} 
If the largest time lag $\Delta_N=O(T_N/N)$ the conditions in Theorem~\ref{th:joint:CLT} become $\lim_{N\to\infty} T_N=\infty$ and $\lim_{N\to\infty} T_N/N =0$ for consistency and
$\lim_{N\to\infty} T_N=\infty$ and $\lim_{N\to\infty} T_N^2/N =0$ for asymptotic normality.
\end{remark}

The next result states that, for fixed time horizon, in high frequency, the estimator from discrete observations converges, {with an ``anomalous'' speed}, towards the  estimator from continuous observations.
Let $Y \colon \Omega \times [0,\infty) \to \RR$ be a semi-martingale, let ${r}\in \RR$, and let $T\in [0,\infty)$. Then we recall that
\begin{equation}\label{eq:def:local}
    L_T^{r}(Y)=\lim_{\varepsilon\to 0}\frac{1}{2\varepsilon}\int_0^T
    \ind{\{-\varepsilon\leq Y_s-{r}\leq \varepsilon\}}d\langle Y\rangle_s
\end{equation}
is the symmetric local time of $Y$ at ${r}$, up to time $T$.

\begin{theorem}\label{th:disc}
    Let $T\in (0,\infty)$ be fixed. 
\begin{enumerate}[i)]
\item \label{th3:item:QMLE}
For every $N\in \NN$, the likelihood $G_{T,N}(a_+,b_+, a_-b_-)$ and the
quasi-likelihood $\Lambda_{T,N}(a_+,b_+, a_-b_-)$ are both maximal at 
$(
    \widehat{a}^+_{T,N},\widehat{b}^+_{T,N}
    ,\widehat{a}^-_{T,N},\widehat{b}^-_{T,N} )
$
given in \eqref{eq:QMLE_dis_time}.
\item \label{th3:item:LLN}
It holds that
$
(
    \widehat{a}^+_{T,N},\widehat{b}^+_{T,N}
    ,\widehat{a}^-_{T,N},\widehat{b}^-_{T,N} )
\convp[N\to\infty]
   ( \alpha^+_{T},\beta^+_{T}
    ,\alpha^-_{T},\beta^-_{T} )
$
and
\begin{equation} \label{eq:CLT:T}
\begin{split}
& N^{1/4}
  \Big(
	( \widehat{a}^+_{T,N},\widehat{b}^+_{T,N}, \widehat{a}^-_{T,N},\widehat{b}^-_{T,N} ) - ( \alpha^+_{T},\beta^+_{T} ,\alpha^-_{T},\beta^-_{T} ) \Big) 
	 \\
	&     
	\xrightarrow[N\to \infty]{\mathrm{stably}}
	 \sqrt{\frac{4 \sqrt{T}}{3 \sqrt{2 \pi}}  \frac{\sigma_-^2 +\sigma_+^2}{\sigma_-+\sigma_+}} \left( 	
	\frac{  \cQ^{+,1}_{T} - r \cQ^{+,0}_{T}			
	}{ 
\cQ^{+,0}_{T} \cQ^{+,2}_{T}	-(\cQ^{+,1}_{T})^2}
	,
	\frac{
\cQ^{+,2}_{T} - r \cQ^{+,1}_{T} }{
\cQ^{+,0}_{T} \cQ^{+,2}_{T}-(\cQ^{+,1}_{T})^2
}
	,
	\right.
	\\
	&\qquad \qquad \left. 
	 - \frac{\cQ^{-,1}_{T} - r \cQ^{-,0}_{T} 		
	}{ 
\cQ^{-,0}_{T} \cQ^{-,2}_{T}	-(\cQ^{-,1}_{T})^2}
	,
 - \frac{
\cQ^{-,2}_{T}- r\cQ^{-,1}_{T}}{
\cQ^{-,0}_{T} \cQ^{-,2}_{T}-(\cQ^{-,1}_{T})^2
	}
 \right) B_{L^r_T(X)} 
\end{split}
\end{equation}
with $B$ Brownian motion independent of $X$ {and $L^r_T(X)$ symmetric local time of $X$ at $r$, up to time $T$ (see~\eqref{eq:def:local})}.
\end{enumerate}
\end{theorem}

\begin{remark}
The right hand side of \eqref{eq:CLT:T} has the same law as 
\[
\begin{split}
 \sqrt{\frac{4 \sqrt{T}}{3 \sqrt{2 \pi}}  \frac{\sigma_-^2 +\sigma_+^2}{\sigma_-+\sigma_+}}
\begin{pmatrix}
 \begin{pmatrix}
\cQ^{+,2}_T & -\cQ^{+,1}_T  \\
-\cQ^{+,1}_T & \cQ^{+,0}_T 
\end{pmatrix}^{\!-1}
&
\begin{pmatrix}
0&0 \\
0&0
\end{pmatrix}
\\
\begin{pmatrix}
0&0 \\
0&0
\end{pmatrix}
&
\begin{pmatrix}
 \cQ^{-,2}_T & -\cQ^{-,1}_T  \\
-\cQ^{-,1}_T & \cQ^{-,0}_T  
\end{pmatrix}^{\! -1}
	\end{pmatrix}
\begin{pmatrix}
-r \\
1 \\
r \\
-1 
\end{pmatrix} \sqrt{L^r_T(X)} B_{1}.
\end{split}
\]
\end{remark}

\begin{remark} \label{local:time:important} 
{One usually expects such discretizations to converge with speed $\sqrt{N}$. In this case, the lower speed of convergence is due to the discontinuity in the coefficients, and appears in connection with the local time. Indeed, the asymptotic behavior of the estimators is intrinsically related to the one of the local time of the process at the threshold. 
More precisely the difference $\cM^{\pm,m}_{T,N}-\cM^{\pm,m}_{T}$, $m=0,1$ can be rewritten involving terms $L^r_{T,N} - L^r_T (X)$, where $L^r_{T,N}$ is the following approximation of the local time from discrete time observations
\begin{equation} \label{def:LT:discr}
	L^r_{T,N}:= 2
	 \sum_{i=0}^{N-1} \ind{\{(X_{i T/N}-r)(X_{(i+1)T/N}-r)<0\}}|X_{(i+1)T/N}-r|
\end{equation}
for $N\in \NN$ 
(see equation~\eqref{eq:M:local:equal} for a more precise statement).}
\end{remark}

\begin{remark}[The skew OU process] \label{rem:skewOU}
{Let us consider the solution to the following SDE involving the local time
\begin{equation}  \label{eq:AffSDOBM}
    Y_t=Y_0+\int_0^t \bar{\sigma}(Y_s)\vd W_s+\int_0^t \left( \bar{b}(Y_s) - \bar{a}(Y_s) \, Y_s \right) \vd s + \bar{\beta} L_t^{\bar{r}}(Y) , \quad t\geq 0,
\end{equation}
with $\bar{\beta}\in (-1,1)$ and piecewise constant functions $\bar{\sigma}$, $\bar{a}$, $\bar{b}$ possibly discontinuous at the threshold $\bar{r}\in \RR$, as in~\eqref{sigmaDOBM} and \eqref{SETvas}.

\cite{XingZhaoLi2020} 
assume $\bar{\beta}$ and $\bar{\sigma}$ known and
consider drift parameters estimation for $Y$, based on discrete observations, in the case of constant $\bar{\sigma}, \bar{a}, \bar{b}$ coefficients and local time at $0$. In this setting, $Y$ is referred to as ``skew OU process'' ({see also \citep{fengparameter}}).

Consider now the more general case of $\bar{\sigma}, \bar{a}, \bar{b}$ as in~\eqref{sigmaDOBM} and \eqref{SETvas}. If  we assume that only $\bar{\beta}$ is known,
all the results in Section \ref{sec:mainresults} on drift estimation for $X$ hold similarly for drift estimation of $Y$.
This follows from the fact that a simple transformation allows us to reduce the skew OU to a threshold OU with threshold at $0$, getting rid of the local time in the dynamics. 
 }
\end{remark}

\begin{remark}[The threshold CIR process]
\cite{su2015,su2017} and \cite{YU2020} consider diffusion processes with drift as in \eqref{SETvas}, with more flexibility on the diffusion coefficient $\sigma(\cdot)$, so that their analysis also applies to the process as in \eqref{eq:AffDOBM} with 
	\begin{equation}
    \label{sigmaCIR}
    \sigma(x)= 
	  \sigma_+ \sqrt{x} \, \ind{\{x\geq r\}} + \sigma_- \sqrt{x} \, \ind{\{x < r\}}. 
\end{equation}
We refer here to this process by threshold CIR (Cox-Ingersoll-Ross). In these works, the proposed estimators are always QMLE, maximising \eqref{QMLE:eq0}, that does not depend on the diffusivity $\sigma(\cdot)$. Here, with our (more restrictive) piecewise-constant choice for the diffusivity, we are able to show that the considered estimator is a genuine MLE. 
In our setting, we expect a result analogous to Theorem \ref{th:continuous} to apply to the QMLE for the threshold CIR as well, but with a  less explicit limit Gaussian law in the CLT, cf. also \citep{su2015}. 
On the other hand, the proof of the discrete time Theorem \ref{th:joint:CLT} makes use of bounds on hitting times for the OU process. The corresponding result for the threshold CIR process does not seem to be a trivial extension.
\end{remark}

\section{Threshold estimation, testing and interest rates}\label{sec:numerics}
We simulate the threshold OU process using the Euler scheme \citep{BOKIL} (an alternative approach for simulation consists in discretizing space instead of time, cf. \citep{cui}) and use the estimator based on discrete observations. The implementation has been done using~\texttt{R}. Parameters are as in Table~\ref{table:simulation_parameters}.
\begin{table}[htp]
\centering
\begin{tabular}
{c|c|c|c|c|c|c}%
 $r$ &
 $b_-$ & $b_+$ & 
 $a_-$ & $a_+$ &
$\sigma_-$ & $\sigma_+$ 
  \\\hline
$0.01$ &
 $-0.002$ & $0.003$ & 
 $0.1$ & $0.11$ &
$0.011$ & $0.01$ 
\end{tabular}
\caption{Simulations parameters.
\label{table:simulation_parameters}}
\end{table}
In Figure~\ref{fig:CLT} we see an example of the CLT in Theorem \ref{th:joint:CLT}. 
\begin{figure}[ht!]
\centering
	\includegraphics[scale=.35]{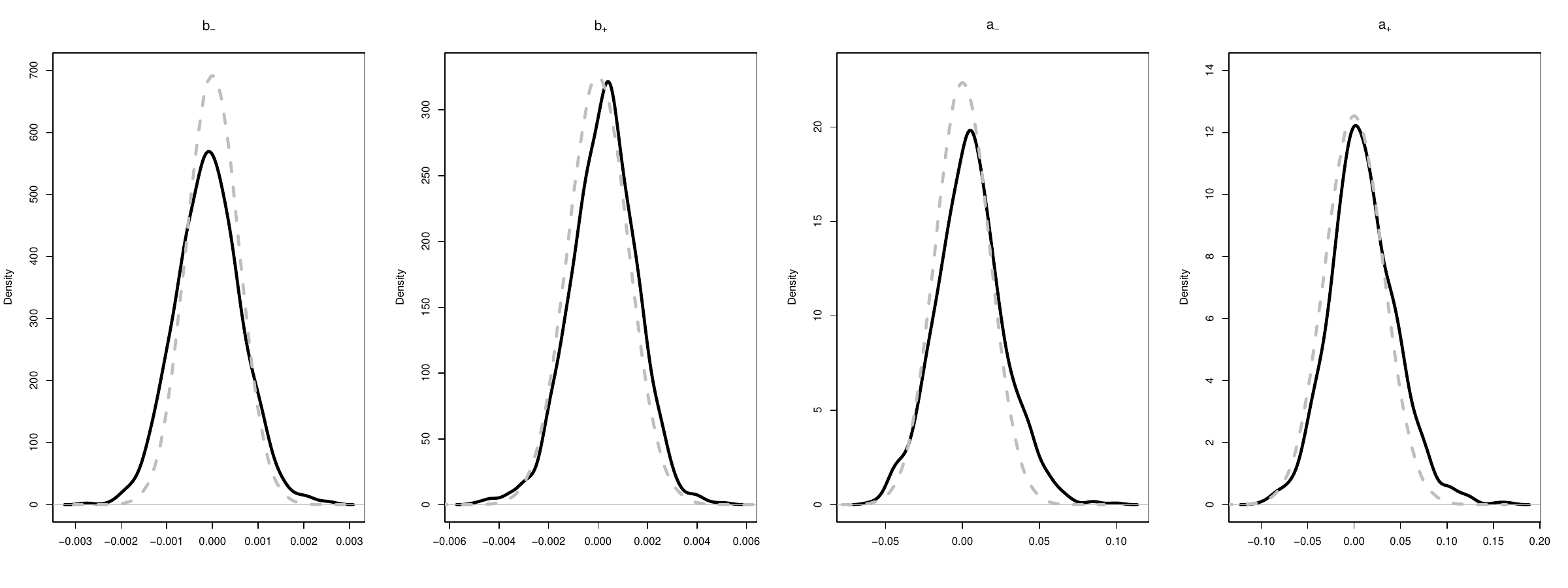}
\caption{CLT in Theorem~\ref{th:joint:CLT}.\eqref{th2:item:CLT}, with parameters as in Table~\ref{table:simulation_parameters}. We plot the theoretical distribution of the estimation error (dashed line) and compare with
the distribution of the error on $n=10^3$ trajectories, with  $T=10^3$ and $N=10^6$ observations on each trajectory.
  }
\label{fig:CLT}
\end{figure}
In Table \ref{tab:metrics:simulation} we show mean, standard deviation and mean squared error of the estimators on simulated trajectories.
\begin{table}[htp]
\centering
\begin{tabular}
{c|c|c|c|c}%
 parameter & $b_-$ & $b_+$ & $a_-$ & $a_+$ 
 \\\hline
mean & $-0.00204$ &
$0.00318$ &
$0.105$ &
$0.119$ 
 \\\hline
MSE &
$5.10\times 10^{-7}$ &
$1.62\times 10^{-6}$ &
$0.000526$ &
$0.00130$ 
 \\\hline
simulated sd &
$0.000713$ &
$0.00126$ &
$0.0223$ &
$0.0349$ 
 \\\hline
predicted sd &
$0.000575$ &
$0.00122$ &
$0.0178$ &
$0.0317$
\end{tabular}
\caption{
Mean, mean squared error (MSE) and standard deviation (simulated sd) of the 
(Q)MLE estimators in \eqref{eq:QMLE_dis_time}, with parameters as in Table \ref{table:simulation_parameters}, on $n = 10^{3}$ trajectories, with $T = 10^{3}$ and $N = 10^{6}$ observations on each trajectory. The ``predicted sd'' is the sd predicted by the Gaussian CLT in Theorem \ref{th:joint:CLT}.\label{tab:metrics:simulation}}
\end{table}
In Figure~\ref{fig:disc} we see an example of the convergence in Theorem~\ref{th:disc}, using that~\eqref{eq:CLT:T} can be rewritten as
\begin{small}
\begin{equation} \label{eq:CLT:Tbis}
\begin{split}
 N^{1/4}
  \Big(
	\frac{ \widehat{a}^\pm_{T,N} - \alpha^\pm_{T}}{\cQ^{\pm,1}_{T} - r \cQ^{\pm,0}_{T}} , \frac{\widehat{b}^\pm_{T,N} -\beta^\pm_{T}}{\cQ^{\pm,2}_{T} - r \cQ^{\pm,1}_{T}}  \Big)  
	\sqrt{\frac{3 \sqrt{2 \pi}}{4 \sqrt{T} L^r_T(X)}  \frac{\sigma_-+\sigma_+}{\sigma_-^2 +\sigma_+^2}} 
	(\cQ^{\pm,0}_{T} \cQ^{\pm,2}_{T}-(\cQ^{\pm,1}_{T})^2)\\
	\xrightarrow[N\to \infty]{\mathrm{stably}}
	\pm  B_{1},
	\mbox{ for } \pm\in \{-,+\}.
\end{split}
\end{equation}
\end{small}
To estimate the local time $L^r_T(X)$ and the occupation times $\cQ^{\pm,i}_T$, we use the discrete time approximations in~\eqref{def:LT:discr} and~\eqref{def:M:Q:disc}.
\begin{figure}[ht!]
\centering
	\includegraphics[scale=.4]{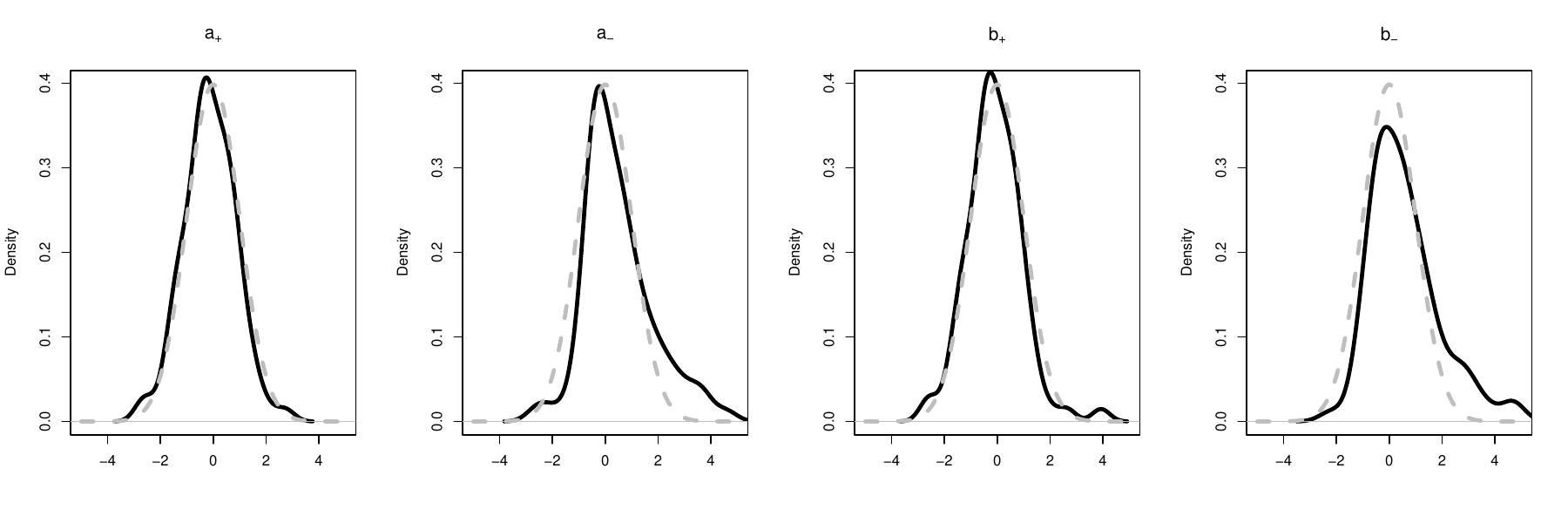}
\caption{Convergence in Theorem~\ref{th:disc}, with parameters as in Table~\ref{table:simulation_parameters}. We compare on $n=100$ trajectories the distribution of the left hand side of~\eqref{eq:CLT:Tbis}, where $T=1$ and $N=500$ discrete observations, with a standard Gaussian (dashed line).
  }
\label{fig:disc}
\end{figure}
If we want to simulate a stationary version of process \eqref{eq:AffDOBM}, we can simulate $X_0$ using explicit stationary density~\eqref{eq:stat:dist} or running the process until large time $T$ and then using the r.v.~$X_T$ as initial condition. 
In Figure~\ref{fig:invariant_density} we compare the empirical distribution obtained in this way with the theoretical stationary density. 
This constitutes an example of bi-modal stationary distribution (density) with two peaks, corresponding to the two different mean reversion levels. 
\begin{figure}[ht!]
    \centering
	\includegraphics[scale=.4]{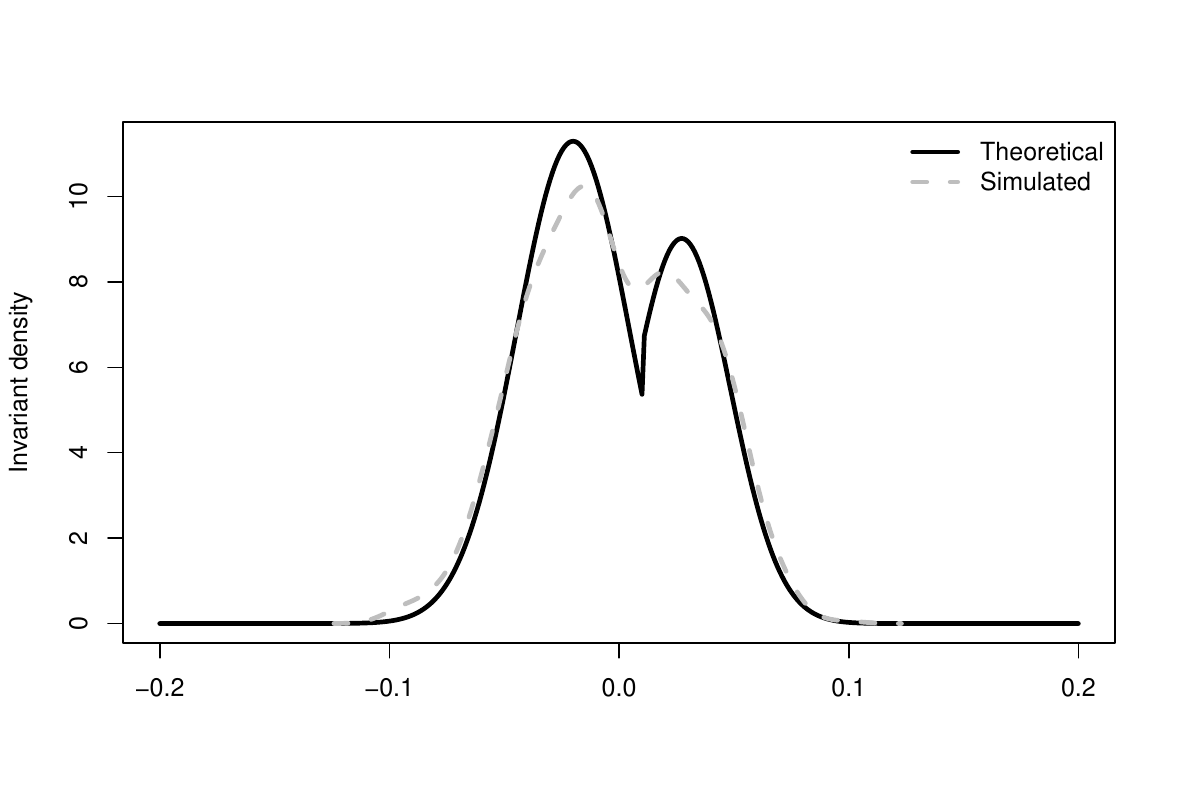}
\caption{Theoretical invariant density in \eqref{eq:stat:dist} vs empirical distribution of $X_T$, with $T=10^3$, with $N=10^6$ discretization steps in Euler scheme, on $n=10^3$ trajectories. Parameters are as in Table~\ref{table:simulation_parameters}.}
    \label{fig:invariant_density}
\end{figure}

\subsection{On threshold estimation}\label{sec:threshold}
The estimation results in Section \ref{sec:mainresults} suppose the previous knowledge of the threshold. In practice, this assumption is not realistic and the threshold $r$ has to be estimated as well.
In \citep{su2015}, threshold  QMLE from continuous observations is shown to be $T$-consistent. We implement here also the analogous threshold MLE, and we directly consider discrete observations starting from the convergence results in Theorem \ref{th:joint:CLT}.

Given $N$ discrete observations of one trajectory up to time $T_N$, we proceed as follows. 
First, for a given threshold $r$, we compute (Q)MLE  $(\widehat{a}^\pm, \widehat{b}^\pm)_{T_N,N}$, and denote this estimator $(\widehat{a}^\pm, \widehat{b}^\pm)_{T_N,N}^r$.
 For each fixed $r$, we can then compute the quasi-likelihood function $\Lambda_{T_N,N}$. We can also compute the likelihood function $G_{T_N,N}$, after estimating $\sigma_\pm$ using the quadratic variation estimators in \citep{LP}.
We take $c$ to be the $\delta$-percentile and $d$ the $1-\delta$ percentile of the observed data (in the implementation we always take $\delta=0.15$ and vary $r$ on a discrete grid). Maximizing now the (quasi-)likelihood function over $r\in[c,d]$ we obtain the (Q)MLE of the threshold, $\widehat{r}$.
The estimator of all the drift parameters is then $(\widehat{r},(\widehat{a}^\pm, \widehat{b}^\pm)_{T_N,N}^{\widehat{r}})$.

We display a sample trajectory in Figure~\ref{fig:traj}, together with the threshold estimated on that trajectory and mean reversion levels. 
\begin{figure}[ht!]
  \centering
	\includegraphics[scale=.41]{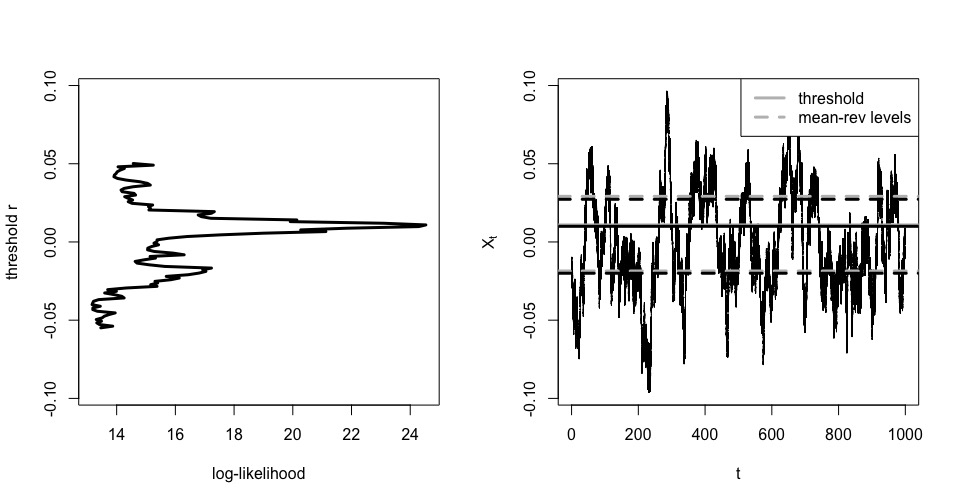}
\caption{A sample trajectory with parameters as in Table~\ref{table:simulation_parameters}, 
$T=10^3$ and $N=10^6$ time steps, and the results of estimation of both threshold and parameters, using MLE. On the left, we show the log-likelihood (on the x-axis) as a function of the threshold (on the y-axis), in order to visualize the procedure for threshold estimation described in Section \ref{sec:threshold}. On the right, we show estimated vs actual threshold level and mean reversion levels $b_-/a_-$ and $b_+/a_+$.}
    \label{fig:traj}
\end{figure}
Estimated parameters are in Table \ref{table:est:par} (cf. with simulation parameters in Table \ref{table:simulation_parameters}).
\begin{footnotesize}
\begin{table}[htp]
\centering
\begin{tabular}
{c|c|c|c|c|c|c|c}%
parameter & $r$ &
 $b_-$ & $b_+$ & 
 $a_-$ & $a_+$ &
$b_-/a_-$ & $b_+/a_+$ 
  \\\hline
(Q)MLE &
$0.0109$ &
 $-0.00222$ & $0.00403$ & 
 $0.119$ & $0.138$ &
$-0.0186$ & $0.0292$ 
  \\\hline
sd &
 &
 $0.000649$ & $0.00129$ & 
 $0.0218$ & $0.0342$ &
 &  
\end{tabular}
\caption{Estimated drift parameters corresponding to Figure~\ref{fig:traj}. Note that in this case, the threshold maximizing MLE and QMLE is the same (on the discrete grid we consider), but this is not necessarily the case. With the same threshold, also estimates for $a_\pm,b_\pm$ are the same, from Theorem~\ref{th:joint:CLT}.\eqref{th2:item:QMLE}.  
We also show the standard deviation of such estimators according to the Gaussian CLT in Theorem \ref{th:joint:CLT}.
\label{table:est:par}}
\end{table}
\end{footnotesize}
Note that MLE and QMLE give the same parameter estimates once the threshold is fixed (Theorem~\ref{th:joint:CLT}.\eqref{th2:item:QMLE}). However, when maximizing also over the choice of the threshold, the MLE can also account of a possible change in the volatility, and this may give a different choice of the threshold. {The model with different volatilities (SET Vasicek) is used by \cite{interestrate}.  \cite{su2015,su2017} use the QMLE, so their drift estimator does not account of possible changes in the volatility.}

\subsection{Testing for threshold}\label{sec:test}

We aim to test the presence of a threshold in the diffusion dynamics. \cite{su2017} propose a test for the presence of a threshold based on quasi-likelihood ratio. Here, we derive a test from the CLT in Theorem~\ref{th:joint:CLT}.\eqref{th2:item:CLT}; therefore, we suppose that its assumptions are satisfied. 
Moreover, we assume that the threshold parameter $r$ is given. In applications, a natural choice for $r$ will be the (Q)MLE $\hat{r}$. 
With fixed threshold, 
we can estimate the drift parameters obtaining 
$(\widehat{a}^\pm_{T_N,N}, \widehat{b}^\pm_{T_N,N})$. From Theorem~\ref{th:joint:CLT}.\eqref{th2:item:CLT}, if $T_N^2/N$ goes to 0 as $N\to\infty$,
\begin{equation}\label{eq:test:global}
\begin{split}&	\sqrt{T_N}
	\begin{pmatrix}
 		(\widehat{a}^+_{T_N,N} - \widehat{a}^-_{T_N,N})-(a_+- a_-) , (\widehat{b}^+_{T_N,N} - \widehat{b}^-_{T_N,N})-(b_+- b_-) 
	\end{pmatrix} \\
&	\convsl[N\to\infty] 
	N^{+} - N^{-}
	\end{split}
\end{equation} 
which is a centered Gaussian vector with covariance matrix given by $\Sigma:=\sigma_+^2 \Gamma_+^{-1} + \sigma_-^2 \Gamma_-^{-1}$,
 invertible from Cauchy-Schwarz inequality. 
The inverse matrix $\Sigma^{-1}$ can be expressed as a function of $\sigma_\pm$ and $\cQ^{\pm,i}_\infty$. 
Note that
$\sigma_\pm$ can be estimated from one observed trajectory using quadratic variation as in \citep{LP} and $\cQ^{\pm,i}_\infty$ can be estimated computing $\frac1{T_N} \cQ_{T_N,N}^{\pm,i}$ as Riemann sums on the observed trajectory, from \eqref{th:erg:covariance}. 
We denote $\widehat \Sigma^{-1}$ the estimate of $\Sigma^{-1}$ obtained from these estimations.

To test for the presence of a threshold in the drift we consider hypothesis 
\[
	\begin{cases}
		H_0 \colon \text{ Null hypothesis } 				& (a_+,b_+)=(a_-,b_-)\\
		H_1 \colon \text{ Alternative hypothesis } & (a_+,b_+)\neq(a_-,b_-).
	\end{cases}
\]
Under the null hypothesis the statistics
\[ 
	T_N  	
	\begin{pmatrix}
 		(\widehat{a}^+_{T_N,N} - \widehat{a}^-_{T_N,N}) , (\widehat{b}^+_{T_N,N} - \widehat{b}^-_{T_N,N})
	\end{pmatrix}  
	\widehat \Sigma^{-1}
	\begin{pmatrix}
 		(\widehat{a}^+_{T_N,N} - \widehat{a}^-_{T_N,N}) \\ (\widehat{b}^+_{T_N,N} - \widehat{b}^-_{T_N,N})
	\end{pmatrix}  
\]
converges to a $\chi^2$ distribution with 2 degrees of freedom.
We reject $H_0$ if the statistics is larger than $q_\alpha$, where $q_\alpha$ is the quantile of a $\chi^2$ distribution with two degrees of freedom such that $\PP(\chi^2_2 \geq q_\alpha)=\alpha$.

To conclude, note that~\eqref{eq:test:global} similarly allows to test separately the presence of a threshold in $a(\cdot)$ or in $b(\cdot)$, i.e.~testing for the presence of a discontinuity in the piecewise linear or in the piecewise constant part of the drift.

\subsection{Interest rate analysis}

We consider the 3 months US Treasury Bill rate, time series of daily closing rate on period Jan 04, 1960 - Apr 29, 2020 (source: Yahoo Finance). 
We perform quasi-maximum and maximum likelihood estimation using \eqref{eq:QMLE_dis_time}, adopting the convention that the ``daily'' time interval is $\vd t = 0.046$ months, while one month is the time unit. The number of observations is $N=15057$, whereas $T\approx 60$ years. We choose as percentile for the search of the threshold $\delta=0.15$. We report both our MLE and QMLE parameters.

We see in Figure \ref{fig:interest_rate} (bottom) that in the case of QMLE our result is consistent with the one in \citep{su2015}, so that the estimation identifies two regimes. One is low rates, with negligible drift, so that in this regime the process is almost a martingale. In the high regime, a stronger reversion to lower rates is ensured by the drift when the rates are very high. When - in Figure \ref{fig:interest_rate} (top) - we use MLE (with $\sigma_\pm$ estimated using quadratic variation), the estimation identifies a low regime corresponding to the period of extremely low rates, with minimal fluctuations, that followed the 2007-2008 financial crisis, whereas almost all the rest of the time series is in the high regime. Volatilities thus estimated are $\sigma_-=0.186$ in the low regime, $\sigma_+=0.453$ in the high regime. 
To compute standard deviations of these estimates we apply \cite[Corollary 3.8.]{LP}, in the form of \cite[Proposition 3.1]{lp1}, with the same justification as after \cite[Proposition 3.1]{lp1}. We get that the standard deviation is
$0.00472$
 for $\sigma_-^{2}$ and $0.0120$ for $\sigma_+^{2}$. 
With this MLE for the drift, the mean reverting effect looks non-negligible both above and below the threshold. 
We note that parameter estimates obtained trough MLE and QMLE are substantially different. This is due to the different choice of the threshold, that in the QMLE does not depend on the behaviour of the volatility, while the MLE is influenced by the volatility as well. For this reason, when using the MLE, one of the two regimes isolated by the threshold only consists of the period of extremely low and stable rates that followed the 2008 financial crisis.

\begin{figure}[ht!]
    \centering
	\includegraphics[scale=.5]{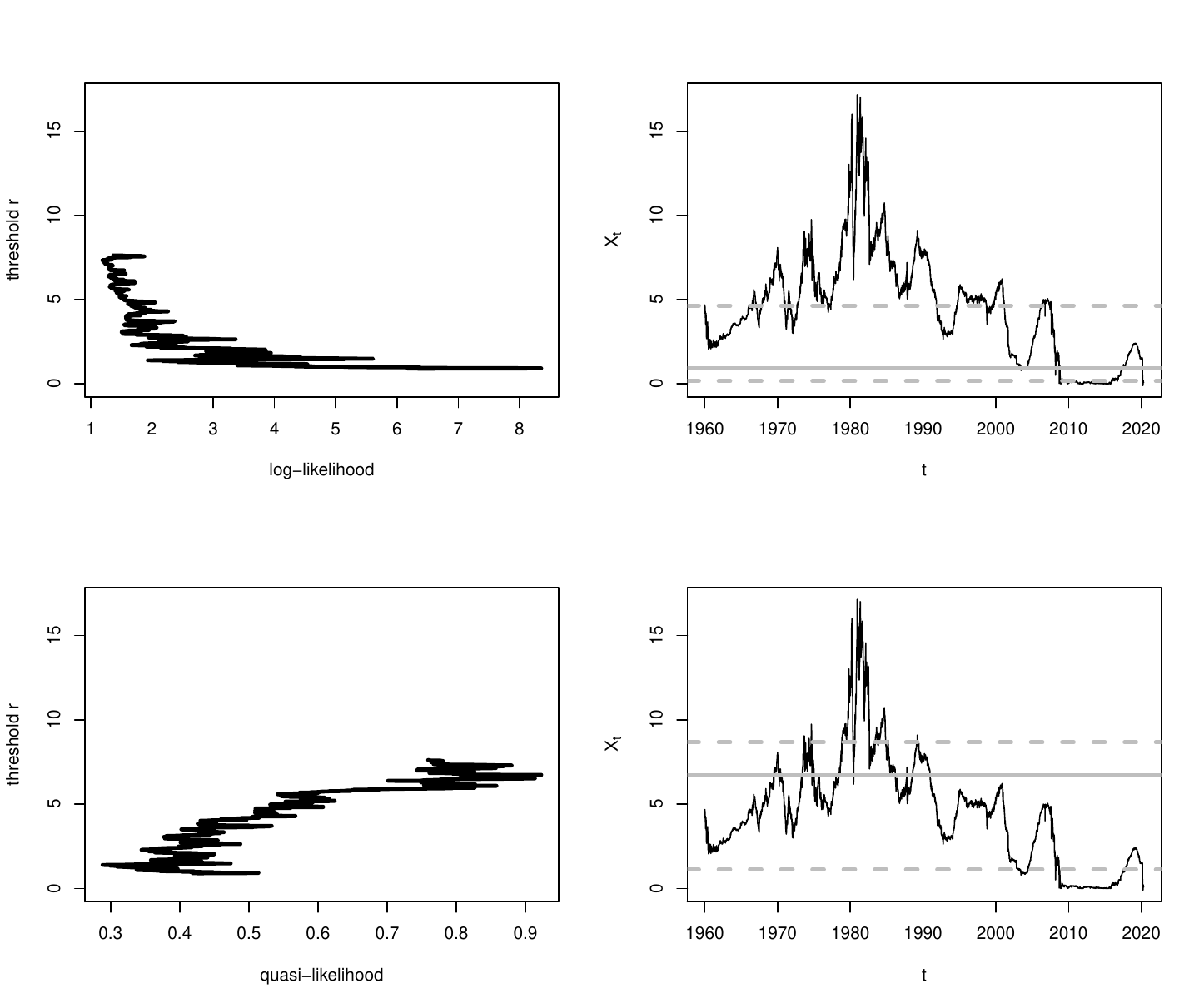}
\caption{3 months US Treasury Bill rate, time series of daily closing rate on period Jan 04, 1960 - Apr 29, 2020.  In the top figure we use the MLE, in the bottom figure  the QMLE. On the right hand side we show log-likelihood and quasi-likelihood as a function of the threshold. On the right hand side we show
estimated threshold levels (solid grey line) and mean reversion levels $b_-/a_-$ and $b_+/a_+$ (dashed grey line).}
    \label{fig:interest_rate}
\end{figure}
\begin{table}[htp]
\centering
\begin{tabular}
{c|c|c|c|c|c|c|c}%
parameter & $r$ &
 $b_-$ & $b_+$ & 
 $a_-$ & $a_+$ &
$b_-/a_-$ & $b_+/a_+$ 
  \\\hline
MLE &
$0.919$ &
 $0.0469$ & $0.0492$ & 
 $0.284$ & $0.0106$ &
$0.165$ & $4.63$ 
  \\\hline
MLE sd &
 &
 $0.0223$ & $0.0402$ & 
 $0.0757$ & $0.00672$ &
 &  
  \\\hline
QMLE &
$6.73$ &
 $0.00131$ & $0.417$ & 
 $0.00115$ & $0.0481$ &
$1.14$ & $8.67$ 
  \\\hline
QMLE sd &
 &
 $0.0341$ & $0.144$ & 
 $0.00877$ & $0.0153$ &
 & 
\end{tabular}
\caption{Estimated drift parameters corresponding to Figure~\ref{fig:interest_rate}. \label{table:estimated_parameters}} and corresponding standard deviation for $b_{\pm}$ and $a_{\pm}$ according to Theorem \ref{th:joint:CLT}
\end{table}
 



{
In this analysis, following \cite{su2015,su2017}, we estimated our model parameters on the whole period 1960-2020.  From an econometric perspective, it is natural to wonder whether it is reasonable to assume the stationarity of the process on such a long time interval.
To address this issue, we consider within the period 2010-2020 five two-years time windows, with daily observations as before. With this choice, $T^2\approx N$, and therefore we expect from Theorem \ref{th:joint:CLT} that the discretization error should be negligible, assuming that $T\approx 2$ years is large enough for the theorem to apply.

With $1\%$ significance level, only in the subperiod Jan 2018-Dec 2019 the $H_0$ hypothesis (absence of a threshold in the parameters) is not rejected. In all other time windows (2010-2011, 2012-2013,  2014-2015, 2016-2017) we conclude that a threshold is present. In Figure \ref{fig:interest_rate_window_shift} we see three examples of estimation in such windows. In order to check whether such test is reliable on time series with such sample sizes, we tried the same procedure (selection of threshold and successive test with $1\%$ significance level) on simulated time series with parameters and sample sizes of the same order as the estimated ones. We found that when there is no threshold present (constant parameters) $H_0$ is rejected $14\%$ of the times, when the threshold is present (non-constant parameters) $H_0$ is rejected $96\%$ of the times, which seems to confirm the validity of the procedure, even in these smaller time windows.
\begin{figure}[ht!]
    \centering
	\includegraphics[scale=.34]{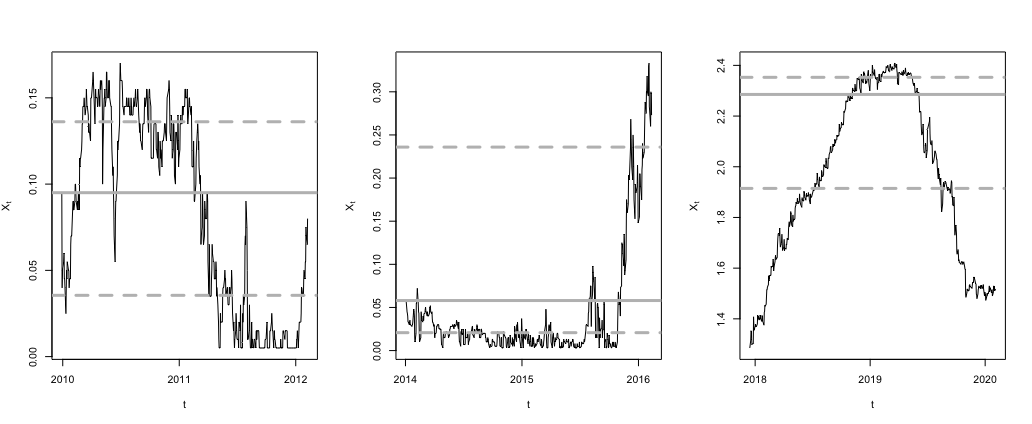}
\caption{3 months US Treasury Bill rate, time series of daily closing rate on periods Jan, 2010 - Dec 2012, Jan, 2014 - Dec 2016, Jan, 2018 - Dec 2020. Estimated threshold level  and mean reversion levels given by QMLE. Each time window consists in $24$ months, with $22$ observations per month. Our test concludes that a threshold in the dynamics is present in every time window, except Jan, 2018 - Dec 2020.
}
    \label{fig:interest_rate_window_shift}
\end{figure}
}

\appendix
\section{Proofs} \label{sec:proofs}

\subsection{The regimes of the process}

In this section, we establish for which values of the coefficients $(a_\pm,b_\pm)$ the process $X$ is (positively or null) recurrent or transient.
Since $X$ is a one-dimensional diffusion it is characterized by two quantities: {\it scale function} $S$ and {\it speed measure}.\\
$X$ is \emph{recurrent} if and only if $\lim_{x\to+\infty} S(x)=+\infty$ and $\lim_{x\to-\infty} S(x)=-\infty$, otherwise it is transient. Moreover a recurrent process is \emph{positive recurrent} if the speed measure is a finite measure, otherwise null recurrent.

The scale density is continuous, unique up to a multiplicative constant, and its derivative satisfies 
$S'(x)=\exp{\!\left( - \int_r^x \frac{2(b(y)- a(y)y)}{(\sigma(y))^2}\vd y\right)}$.
%
It 
follows that
{$X$ is recurrent} if and only if
{
[($a_+>0$ and $b_+\in \RR$) or ($a_+=0$ and  $b_+\leq 0$)] and
[($a_->0$ and $b_-\in \RR$) or ($a_-=0$ and  $b_-\geq 0$)]}.
The complementary leads to transience.

The density of the \emph{speed measure} 
with respect to the Lebesgue measure 
is given by 
$
	m(x) := \frac{2}{(\sigma(x))^2 S'(x)}.
$
It is \emph{discontinuous} if and only if $\sigma^2$ is so.
Assume $X$ is recurrent. The speed measure is a finite measure, and so {$X$ is  positive recurrent}, if and only if
\begin{equation} \label{eq:ergodic:cases}
\begin{split}
& \textrm{
[($a_+>0$ and $b_+\in \RR$) or ($a_+=0$ and  $b_+< 0$)]} 
\\
& \textrm{and [($a_->0$ and $b_-\in \RR)$ or ($a_-=0$ and  $b_-> 0$)].
}
\end{split}
\end{equation}
See Lemma~\ref{lem:ergodic} below.
In these cases, the process is actually \emph{ergodic} and the stationary distribution $\mu$ is equal to the renormalized speed measure: 
\begin{equation} \label{eq:stat:dist}
	\mu(\!\vd x) = \frac{m(x)}{\int_{-\infty}^\infty m(y) \vd y} \vd x.
\end{equation}

\begin{lemma}
\label{lem:ergodic}
Let  $\pm\in\{-,+\}$ and let 
\begin{equation} \label{eq:aux:statdistr}
	\mathfrak{m}_{\pm}:=  \frac{\sqrt{\pi}}{\sigma_\pm \sqrt{a_\pm}} 
	\exp{\!\left(\frac{a_\pm}{\sigma_\pm^2} \left(\frac{b_\pm}{a_\pm} -r \right)^{\! 2} \right)}
	\operatorname{erfc}{\!\left(\mp \frac{\sqrt{a_\pm}}{\sigma_\pm} \left(\frac{b_\pm}{a_\pm}-r\right)\right)}.
\end{equation}
Then
\[
\int_{-\infty}^\infty \ind{\{\pm(y-r) \geq 0\}} m(y) \vd y =
\begin{cases}
+\infty & \text{ if } a_\pm=0 \text{ and } b_\pm=0, \\
\frac{1}{|b_\pm|} & \text{ if } a_\pm=0 \text{ and } \mp b_\pm > 0, \\
\mathfrak{m}_{\pm} & \text{ if } a_\pm>0 \text{ and } b_\pm \in \RR.
\end{cases}
\]
\end{lemma}

\begin{lemma}  \label{eq:Qinfinity2}
Assume the process is ergodic.
Let $\mathfrak{m}_{-}, \mathfrak{m}_{+}$, given by~\eqref{eq:aux:statdistr}, 
$\mathfrak{b}_{\pm}=1/|b_\pm|$, $\pm\in \{-,+\}$, and $\mu$ be the stationary distribution.
For all $i\in \{0,1,2\}$ let 
$\overline{\cQ}^{\pm,i}_\infty$ be the constant such that
$\overline{\cQ}^{\pm,i}_\infty {\overset{\as}{=}} \lim_{t\rightarrow \infty} \frac{{\cQ}^{i,\pm}_t}{t} \in \RR$.
We have the following explicit formulas:
\begin{itemize}
\item if $a_+>0$, $a_- >0$, $b_-,b_+\in \RR$ then
{
\begin{equation*}
\begin{split}
& \overline{\cQ}^{\pm,0}_\infty = \frac{\mathfrak{m}_{\pm}}{\mathfrak{m}_{+}+\mathfrak{m}_{-}}, 
\quad  
\overline{\cQ}^{\pm,1}_\infty = \frac{1}{\mathfrak{m}_{+}+\mathfrak{m}_{-}} \left( \frac{b_\pm}{a_\pm} {\mathfrak{m}_{\pm}} \pm \frac{1}{a_\pm} \right), \quad \text{and} \\
& 
\overline{\cQ}^{\pm,2}_\infty = \frac{1}{\mathfrak{m}_{+}+\mathfrak{m}_{-}}
\left(
	 \left( \frac{b_\pm^2}{a_\pm^2} +\frac{\sigma_\pm^2}{2 a_\pm} \right) \mathfrak{m}_{\pm} 
	\pm \left(\frac{b_\pm}{a_\pm}+r\right) \frac{1}{a_\pm}\right);
\end{split}
\end{equation*}
}
\item if $a_+=0$, $a_-=0$, $b_+<0$, $b_->0$ then
{
\begin{equation*}
\begin{split}
& \overline{\cQ}^{\pm,0}_\infty = \frac{\mathfrak{b}_{\pm}}{\mathfrak{b}_{+}+\mathfrak{b}_{-}}, 
\quad  
\overline{\cQ}^{\pm,1}_\infty = \frac{{\mathfrak{b}_{\pm}}}{\mathfrak{b}_{+}+\mathfrak{b}_{-}}  \left( r  \pm \frac{\sigma_\pm^{2}}{2} {\mathfrak{b}_{\pm}} \right), \quad \text{and} \\
& 
\overline{\cQ}^{\pm,2}_\infty = \frac{{\mathfrak{b}_{\pm}}}{\mathfrak{b}_{+}+\mathfrak{b}_{-}}
\left(
	 r^2 \pm r {\sigma_\pm^{2}} {\mathfrak{b}_{\pm}} + \frac{\sigma_\pm^{4}}{2} ({\mathfrak{b}_{\pm}})^2\right);
\end{split}
\end{equation*}
}
\item $a_+>0$, $b_+\in \RR$, $a_-=0$, $b_->0$ then
{
\begin{equation*}
\begin{split}
& \overline{\cQ}^{+,0}_\infty = \frac{\mathfrak{m}_{+}}{\mathfrak{m}_{+}+\mathfrak{b}_{-}}, 
\quad 
\overline{\cQ}^{-,0}_\infty = \frac{\mathfrak{b}_{-}}{\mathfrak{m}_{+}+\mathfrak{b}_{-}}
\\
&  
\overline{\cQ}^{+,1}_\infty = \frac{1}{\mathfrak{m}_{+}+\mathfrak{b}_{-}} \left( \frac{b_+}{a_+} {\mathfrak{m}_{+}} + \frac{1}{a_+} \right), \quad 
\overline{\cQ}^{-,1}_\infty = \frac{{\mathfrak{b}_{-}}}{\mathfrak{m}_{+}+\mathfrak{b}_{-}} \left( r  - \frac{\sigma_-^{2}}{2} {\mathfrak{b}_{-}} \right), \\
& 
\overline{\cQ}^{+,2}_\infty = \frac{1}{\mathfrak{m}_{+}+\mathfrak{b}_{-}}
\left(
	 \left( \frac{b_+^2}{a_+^2} +\frac{\sigma_+^2}{2 a_+} \right) \mathfrak{m}_{+} 
	+ \left(\frac{b_+}{a_+}+r\right) \frac{1}{a_+}\right), \quad \text{and} \\
& 
\overline{\cQ}^{-,2}_\infty = \frac{{\mathfrak{b}_{-}}}{\mathfrak{m}_{+}+\mathfrak{b}_{-}}
\left(
	 r^2 - r {\sigma_-^{2}} {\mathfrak{b}_{-}} + \frac{\sigma_-^{4}}{2} ({\mathfrak{b}_{-}})^2\right);
\end{split}
\end{equation*}
}
\item $a_+=0$, $b_+<0$, $a_->0$, $b_-\in \RR$ then
{
\begin{equation*}
\begin{split}
& \overline{\cQ}^{+,0}_\infty = \frac{\mathfrak{b}_{+}}{\mathfrak{b}_{+}+\mathfrak{m}_{-}}, 
\quad 
\overline{\cQ}^{-,0}_\infty = \frac{\mathfrak{m}_{-}}{\mathfrak{b}_{+}+\mathfrak{m}_{-}}
\\
&  
\overline{\cQ}^{+,1}_\infty =  \frac{{\mathfrak{b}_{+}}}{\mathfrak{b}_{+}+\mathfrak{m}_{-}}  \left( r  + \frac{\sigma_+^{2}}{2} {\mathfrak{b}_{+}} \right), \quad 
\overline{\cQ}^{-,1}_\infty = \frac{1}{\mathfrak{b}_{+}+\mathfrak{m}_{-}} \left( \frac{b_-}{a_-} {\mathfrak{m}_{-}} - \frac{1}{a_-} \right), \\
& 
\overline{\cQ}^{+,2}_\infty = \frac{{\mathfrak{b}_{+}}}{\mathfrak{b}_{+}+\mathfrak{m}_{-}}
\left(
	 r^2 + r {\sigma_+^{2}} {\mathfrak{b}_{+}} + \frac{\sigma_+^{4}}{2} ({\mathfrak{b}_{+}})^2\right), \quad \text{and} \\
& 
\overline{\cQ}^{-,2}_\infty = \frac{1}{\mathfrak{b}_{+}+\mathfrak{m}_{-}}
\left(
	 \left( \frac{b_-^2}{a_-^2} +\frac{\sigma_-^2}{2 a_-} \right) \mathfrak{m}_{-} 
	- \left(\frac{b_-}{a_-}+r\right) \frac{1}{a_-}\right).
\end{split}
\end{equation*}
}
\end{itemize}
\end{lemma}
\subsection{Proof of Theorem~\ref{th:continuous}}

\begin{proof}[Proof of Item~\eqref{th1:item:QMLE} of Theorem~\ref{th:continuous}]
Let $\theta :=(a_+, b_+, a_-, b_- )$. It holds that
\begin{equation} \label{QMLE:eq1}
\begin{split}
	\Lambda_T (\theta) 
	& = \sum_{\pm\in \{-,+\}} \bigg( b_\pm \cM_T^{\pm,0} -
	a_\pm  \cM_T^{\pm,1} 
	-\frac{1}{2}\bigg(
	b_\pm^2 \cQ_T^{\pm,0} 	 +
	a_\pm^2 \cQ_T^{\pm,2}
	-2 a_\pm b_\pm  \cQ_T^{\pm,1}
	\bigg) \bigg) .
	\end{split}
\end{equation}
To find the maximum we compute the derivatives with respect to $a_\pm, b_\pm$ and observe that the gradient has a unique singular point given by~\eqref{eq:QMLE_ct} and the Hessian is negative definite.

Moreover the fact that $\partial_{a_\pm}  \Lambda_T=\sigma_\pm^2 \partial_{a_\pm} \log G_T$ 
and 
$\partial_{b_\pm}  \Lambda_T=\sigma_\pm^2 \partial_{b_\pm} \log G_T$
shows that the MLE for the drift parameters will be the same as the QMLE, i.e.~\eqref{eq:QMLE_ct}.
\end{proof}


In order to study the asymptotic behavior of the estimator we introduce a different expression for the estimators in \eqref{eq:QMLE_ct} based on the following notation.
Given $T\in (0,\infty)$, let 
\begin{equation*}
Q^{\pm,i}_T:= \int_0^T |X_s-r|^i \ind{\{ \pm (X_s -r) \geq 0\}}\vd s 
\quad \text{ and } \quad  
M_T^{\pm,j} :=
		\sigma_\pm \int_0^T |X_s -r|^{j-1} \ind{\{ \pm (X_s- r)\geq 0 \}} \vd W_s
\end{equation*}
with $i\in \{0,1,2\}$, $j\in \{1,2\}$. Observe that~\eqref{eq:AffDOBM} yields for $i\in \{0,1\}$:
\begin{equation} \label{eq:equality:Ms}
\begin{split}
 	\cM_T^{\pm,i}
	= r^i M^{\pm,1} \pm i M^{\pm,2} + b_\pm \cQ^{\pm,i}_T 
- a_\pm \cQ^{\pm,i+1}_T.
\end{split}
\end{equation}
Moreover,
$\cQ^{\pm,0}_T=Q^{\pm,0}_T$, 
$
	 \cQ^{\pm,1}_T = \pm \left( Q^{\pm,1}_T \pm r Q^{\pm,0}_T \right)
$, 
and 
$\cQ^{\pm,2}_T = Q^{\pm,2}_T \pm 2 r Q^{\pm,1}_T + r^2 Q^{\pm,0}_T$.

\begin{lemma} \label{lem:QMLE:ab}
Let $T\in (0,\infty)$.
The MLE and QMLE can be expressed as
\begin{equation} \label{eq:MLE:asy}	
\begin{cases}
	\alpha^\pm_T 
	= 
		a_\pm 
		\pm	  
		
		\tfrac{Q_T^{\pm,1}{M_T^{\pm,1}} - Q_T^{\pm,0} {M_T^{\pm,2}}}{Q_T^{\pm,2} Q^{\pm,0}_T - (Q^{\pm,1}_T)^2}
		\\
	\beta^\pm_T 
	 =
		b_\pm 
		+ 
		\tfrac{ (Q_T^{\pm,2} \pm r Q_T^{\pm,1}){M_T^{\pm,1}} - (Q_T^{\pm,1} \pm r Q_T^{\pm,0}) {M_T^{\pm,2}} }{Q_T^{\pm,2} Q^{\pm,0}_T - (Q^{\pm,1}_T)^2}
		\!,
	\end{cases}
\end{equation}
that can be rewritten as
\begin{equation} \label{eq:MLE:asy2}	
\begin{pmatrix}
\alpha^\pm_T  \\ \beta^\pm_T
\end{pmatrix}
= \begin{pmatrix}
	a_\pm  \\  b_\pm
\end{pmatrix}
 + 
 \begin{pmatrix}
	0 & \mp 1
	\\
	1 & \mp r
\end{pmatrix}
\begin{pmatrix}
	Q^{\pm,0}_T & Q^{\pm,1}_T
	\\
	Q^{\pm,1}_T & Q^{\pm,2}_T
\end{pmatrix}^{\! -1}
\begin{pmatrix}
	M_T^{\pm,1}  \\ M_T^{\pm,2}
\end{pmatrix}.
\end{equation}
\end{lemma}
\begin{proof}
Note that $\cQ^{\pm,0}_T \cQ^{\pm,2}_T - (\cQ^{\pm,1}_T)^2 = Q^{\pm,0}_T Q^{\pm,2}_T - (Q^{\pm,1}_T)^2$ which is $\PP$-a.s.~positive by Cauchy-Schwarz.
This and replacing the equalities~\eqref{eq:equality:Ms} in \eqref{eq:QMLE_ct} completes the proof.
\end{proof}

\begin{proof}[Proof of Item~\eqref{th1:item:LLN} of Theorem~\ref{th:continuous}] 
Follows from combining equation~\eqref{eq:MLE:asy} in Lemma~\ref{lem:QMLE:ab} with \cite[Theorem~1, p.150]{lepingle} and Lemma~\ref{eq:Qinfinity2}.
\end{proof}

\begin{proof}[Proof of Item~\eqref{th1:item:CLT} of Theorem~\ref{th:continuous}]
Follows from Lemma~2, \eqref{th:erg:covariance}, and Theorem~2.2 in \citep{crimaldi}.
\end{proof}

\begin{proof}[Proof of Item~\eqref{th1:item:LAN} of Theorem~\ref{th:continuous}] 
By~\eqref{eq:equality:Ms} it holds
\[
\begin{split}
	& \log \frac{ G_T({a}_++ \frac{1}{\sqrt{T}}\Delta a_+, {b}_++ \frac{1}{\sqrt{T}} \Delta b_+,{a}_- + \frac{1}{\sqrt{T}} \Delta a_-, {b}_- + \frac{1}{\sqrt{T}} \Delta b_-)}
	{G_T({a}_+, {b}_+,{a}_-, {b}_-)}
	\\
	& ={\scriptsize  \sum_{\pm \in \{+,-\}}  \left(
	   \frac1{\sqrt{T} \sigma_\pm} \begin{pmatrix} \Delta a_\pm \\  \Delta b_\pm \end{pmatrix}
	 \cdot A^\pm_T 	
	 - \frac{1}{2 T \sigma_\pm^2} 	 
	\begin{pmatrix} \Delta a_\pm \\  \Delta b_\pm \end{pmatrix}
 \cdot \langle A^{\pm}, A^{\pm}\rangle_T  \begin{pmatrix} \Delta a_\pm \\  \Delta b_\pm \end{pmatrix}
\right).}
\end{split}
\]
where
$ 
	A_T^{\pm}	:= \begin{pmatrix} - M^{\pm,1}_T \\ M^{\pm,0}_T\end{pmatrix}.
$
Note that
$
 \left\langle A^\pm, A^\pm \right\rangle_T = 
\begin{pmatrix} 
	 \cQ^{\pm,2}_T  
		& 
	- \cQ^{\pm,1}  _T
		\\
- \cQ^{\pm,1}  _T
			& 
	 \cQ^{\pm,0}  _T
\end{pmatrix}
$
and $\left\langle A^{+} , A^{-} \right\rangle_T=0$.
Lemma~\ref{eq:Qinfinity2} ensures that $\frac1{T} \left\langle A^\pm, A^\pm \right\rangle_T \convas[T\to\infty] \Gamma_\pm $.
This and the same argument as in the proof of Theorem~\ref{th:continuous}.\eqref{th1:item:CLT} show that
\[
	A_T := \frac1{\sqrt{T}}\begin{pmatrix} \sigma_+^{-1} A_T^{+} \\ \sigma_-^{-1} A_T^{-}\end{pmatrix}
\convl[T\to\infty] \mathcal{N}(0,\Gamma)
\]
 and
$
 \left\langle A_T, A_T \right\rangle_T 
= \frac1T {\scriptsize \begin{pmatrix}  \sigma_+^{-2} \left\langle A^{+}, A^{+}\right\rangle_T & 0 \\  0 & \sigma_-^{-2} \left\langle A^{-}, A^{-}\right\rangle_T\end{pmatrix}}
\convas[T\to\infty] \Gamma.
$
\end{proof}


\subsection{Proof of Theorem~\ref{th:joint:CLT}}

The proof of Item~\eqref{th2:item:QMLE} of Theorem~\ref{th:joint:CLT} is analogous to the proof of Item~\eqref{th1:item:QMLE} of Theorem~\ref{th:continuous}, therefore omitted.
%
%
The proof of Items~\eqref{th2:item:LLN}-\eqref{th2:item:CLT} of Theorem~\ref{th:joint:CLT} follows from Lemma~\ref{lemma:control:diff:disc:cont} below.
Let us be more precise. 
For all $N\in \NN$ it holds
\[
\begin{pmatrix}
 \widehat{a}^\pm_{T_N,N}
-a_\pm 
, \
 \widehat{b}^\pm_{T_N,N}
-b_\pm
\end{pmatrix}
=
\begin{pmatrix}
 \widehat{a}^\pm_{T_N,N}
-\alpha_{T_N}^\pm 
, \
 \widehat{b}^\pm_{T_N,N}
-\beta_{T_N}^\pm
\end{pmatrix}
+\begin{pmatrix}
\alpha_{T_N}^\pm 
-a_\pm 
, \
\beta_{T_N}^\pm-b_\pm
\end{pmatrix}.
\]
The second term of the sum is handled with Theorem~\ref{th:continuous} (more precisely Item~\eqref{th1:item:CLT}) providing the desired limit distribution.
The first instead can be rewritten, using equations~\eqref{eq:QMLE_ct} and~\eqref{eq:QMLE_dis_time}, as an expression which involves only terms of the kind
\[
\left(\frac{\cQ^{\pm,i}_{T_N,N}}{ \cQ^{\pm,0}_{T_N,N}\cQ^{\pm,2}_{T_N,N}	-(\cQ^{\pm,1}_{T_N,N})^2} 
- \frac{\cQ^{\pm,i}_{T_N}}{ \cQ^{\pm,0}_{T_N}\cQ^{\pm,2}_{T_N}	-(\cQ^{\pm,1}_{T_N})^2}  \right) \cM^{\pm,j}_{T_N}
+ \frac{\cQ^{\pm,i}_{T_N,N}  { ( \cM^{\pm,j}_{T_N,N} - \cM^{\pm,j}_{T_N} )} }{ \cQ^{\pm,0}_{T_N,N}\cQ^{\pm,2}_{T_N,N}	-(\cQ^{\pm,1}_{T_N,N})^2}
\] 
for $j \in \{0,1\}$, $i\in \{0,1,2\}$,

Combining Lemma~\ref{lemma:control:diff:disc:cont} with Lemma~\ref{eq:Qinfinity2} and 
Theorem~2.2 in \citep{crimaldi}  
ensures, the consistency of the estimator if $T_N/N \to 0$ as $N\to\infty$, and
if $T_N^2/N \to 0$ as $ N\to\infty$ then it
implies also that 
\[
	\sqrt{T_N} \left(\widehat{a}^\pm_{T_N,N} -\alpha^\pm_{T_N} , \ \widehat{b}^\pm_{T_N,N} -\beta^\pm_{T_N}\right) \convp[N\to\infty] 0.
\]

\begin{lemma}\label{lemma:control:diff:disc:cont}
{Assume the process is ergodic. }
Let $X$ be the solution to \eqref{eq:AffDOBM}, with $X_0 $ distributed as the stationary distribution $\mu$ in~\eqref{eq:stat:dist}, let $\lambda \in \{1,2\}$ be fixed,
and let $(T_N)_{N\in \NN} \subset (0,\infty)$ be a sequence satisfying, as $N \to\infty$, that $T_N\to \infty$ and 
$\lim_{N\to\infty} T_N^{1-1/\lambda} \sqrt{\Delta_N} = 0$ 
where $\Delta_N:=\sup_{k=1,\ldots,N} (t_{k}-t_{k-1})$.
Then for all $m\in \{0,1,2\}$, $j\in\{0,1\}$ it holds
\begin{equation*}
\limsup_{N\to\infty} {T_N^{-\nicefrac1{\lambda}}}
 \EE\left[ | \cQ^{\pm,m}_{T_N} - \cQ^{\pm,m}_{T_N,N} |\right] = 0
 \text{ and } 
\limsup_{N\to\infty} {T_N^{-\nicefrac1{\lambda}}} \EE\left[ |\cM^{\pm,j}_{T_N} - \cM^{\pm,j}_{T_N,N} | \right] = 0
\end{equation*}
where $\cQ^{\pm,m}_{T_N}$, $\cQ^{\pm,m}_{T_N,N}$,
$\cM^{\pm,j}_{T_N}$, $\cM^{\pm,j}_{T_N,N}$ are defined in~\eqref{def:M:Q} and \eqref{def:M:Q:disc}.
\end{lemma}
\begin{proof}
Without loss of generality, we reduce to prove the statement for
threshold $r=0$. Indeed the quantities $\cQ^{\pm,m}_{T_N} - \cQ^{\pm,m}_{T_N,N}$ and $\cM^{\pm,m}_{T_N} - \cM^{\pm,m}_{T_N,N}$ for the process $X$ (with threshold $r$) can be written as linear combination (coefficients depending on $m$ and $r$) of the same quantities for the process $X-r$ 
(which solves~\eqref{eq:AffDOBM} with threshold at $0$ and new drift coefficients $b_\pm - a_\pm r$ and $a_\pm$).
We keep denoting as $b_\pm$ (instead of $b_\pm -a_\pm r$) and $a_\pm$ the drift coefficients.
In this proof we use the round ground notation 
$\lfloor t \rfloor_{\Delta_N} := t_k$ for $t \in [t_k,t_{k+1}) \subseteq [t_k, t_k+\Delta_N]$.
Moreover, without loss of generality, we assume $T_N \leq N$ for all $N\in \NN$.

Let us first note that 
for $m=0,1,2$:
\[
\begin{split}
 &\cQ^{\pm,m}_{T_N} - \cQ^{\pm,m}_{T_N,N}
	 = - \sum_{k=1}^{N} J_{k,N}^{(m)}
	\\
	& = \mp \int_0^{T_N} \sgn(X_{\lfloor t \rfloor_{\Delta_N}}) X_{\lfloor t \rfloor_{\Delta_N}}^m 
\ind{\{X_{\lfloor t \rfloor_{\Delta_N}} X_t<0\}} \vd t
+
\int_{0}^{T_N} (X_t^m - X_{\lfloor t \rfloor_{\Delta_N}}^m) \ind{\{ \pm X_t > 0\}} \vd t
\end{split}
\]
therefore
 \[
 \EE\left[|\cQ^{\pm,m}_{T_N} - \cQ^{\pm,m}_{T_N,N}|\right]\\ 
\leq \int_0^{T_N} \EE\left[ |X_{\lfloor t \rfloor_{\Delta_N}}|^m
\ind{\{X_{\lfloor t \rfloor_{\Delta_N}} X_t<0\}} \right]+ \EE\left[ |X_t^m - X_{\lfloor t \rfloor_{\Delta_N}}^m| \right] \vd t.
\] 
%
%
%
%
%
Analogously, observe that for $m=0,1$ it holds
\[\begin{split}
	\cM^{\pm,m}_{T_N} - \cM^{\pm,m}_{T_N,N}  
	& = \int_0^{T_N}  (X_t^m \ind{\{\pm X_t > 0\}} - X_{\lfloor t \rfloor_{\Delta_N}}^m \ind{\{\pm X_{\lfloor t \rfloor_{\Delta_N}}> 0\}}) (b(X_t)- a(X_t )X_t) \vd t
	\\ 
	& \quad  
	+ \int_0^{T_N}  (X_t^m \ind{\{\pm X_t > 0\}}  - X_{\lfloor t \rfloor_{\Delta_N}}^m \ind{\{\pm X_{\lfloor t \rfloor_{\Delta_N}} > 0\}} ) \sigma(X_t) \vd W_t.
\end{split}
\]
Let us rewrite the integrand as
\[
\begin{split}
	 &X_t^m \ind{\{\pm X_t > 0\}} - X_{\lfloor t \rfloor_{\Delta_N}}^m \ind{\{\pm X_{\lfloor t \rfloor_{\Delta_N}}> 0\}}\\
	& =  (X_t^m - X_{\lfloor t \rfloor_{\Delta_N}}^m) \ind{\{\pm X_t > 0\}} 
- \sgn(X_{\lfloor t \rfloor_{\Delta_N}}) X_{\lfloor t \rfloor_{\Delta_N}}^m \ind{ \{ X_t X_{\lfloor t \rfloor_{\Delta_N}} < 0 \}}.
\end{split}
\]
%
Triangular inequality, H\"older's inequality, and It\^o-isometry imply that
\begin{equation*} 
\begin{split}
	& \EE\left[ |\cM^{\pm,m}_{T_N} - \cM^{\pm,m}_{T_N,N} | \right] \\
	&
	\leq   \int_0^{T_N} \EE\left[ |X_t^m - X_{\lfloor t \rfloor_{\Delta_N}}^m|  (|b_\pm| + a_\pm  |X_{\lfloor t \rfloor_{\Delta_N}} | +  a_\pm |X_t -  X_{\lfloor t \rfloor_{\Delta_N}}|)\right] \vd t
	\\
	& + \int_0^{T_N} \EE\Big[ |X_{\lfloor t \rfloor_{\Delta_N}}|^m \ind{\{X_t X_{\lfloor t \rfloor_{\Delta_N}}< 0\}} ( |b_-| \vee |b_+| +(a_- \vee a_+) \\
&\quad\times	
	( |X_{\lfloor t \rfloor_{\Delta_N}}| + |X_t - X_{\lfloor t \rfloor_{\Delta_N}} | ))\Big] \vd t
	\\ 
	& 
	+ 
	\sqrt 2 (\sigma_-\vee \sigma_+) \left(\int_0^{T_N}  \EE\left[  ( X_t^m - X^m_{\lfloor t \rfloor_{\Delta_N}} )^2 +  X^{2 m}_{\lfloor t \rfloor_{\Delta_N}} \ind{\{X_t X_{\lfloor t \rfloor_{\Delta_N}}< 0\}} \right] \vd t\right)^{\!\nicefrac12}.
\end{split}
\end{equation*}

Hence, the proof of Lemma~\ref{lemma:control:diff:disc:cont}, reduces to prove two inequalities:
\begin{equation} \label{item:cond:2}
\int_{0}^{T_N} \EE\left[ |X_t - X_{\lfloor t \rfloor_{\Delta_N}}|^{j} |X_{\lfloor t \rfloor_{\Delta_N}}|^{m}  \right] \!\vd t 
\text{ is } o(T_N^{\nicefrac1\lambda}) 
\end{equation}
for all $j\in \{1,2,4\}, m\in \{0,1,2\}$, and
\begin{equation} \label{item:cond:1}
	\int_0^{T_N} \EE\left[ |X_{\lfloor t \rfloor_{\Delta_N}}|^m \ind{\{X_{\lfloor t \rfloor_{\Delta_N}} X_t<0\}} \right] \!\vd t 
\text{ is } o(T_N^{\nicefrac1\lambda}) \text{ for } m\in \{0,1,2,3,4\}.
\end{equation}

{\it Step 1}. Given $s\in [0,\infty)$ and 
$t\in [0, \Delta_N]$
we show that for every $j\in\{1,2,4\}$ there exists a constant $C\in (0,\infty)$ depending only on $j, a_\pm, b_\pm, \sigma_\pm$ such that
\begin{equation} \label{th:joint:step3}
\EE\left[ | X_{t+s} -  X_{s} |^j  |  X_{s} \right] \leq C t^{j/2} (1+|X_s|^j).
\end{equation}

Let $\xi_t :=X_{t+s} -X_s$ then 
\[\xi_t = \int_0^t (b(\xi_u+X_s) - a(\xi_u+X_s) X_s) - a(\xi_u+X_s) \xi_u \vd u + \int_0^t \sigma(\xi_u+X_s)  \vd W_u^s \] 
where $W^s$ a Wiener process independent of $\sigma(X_u, u\in [0,s])$.
So, given $X_s$, $\xi$ is an OU with threshold $-X_s$ (since $X$ has threshold 0).
Now, e.g.~\cite[Corollary~2.5]{HuddeHutzenthalerMazzonetto} applied to $\xi$ 
implies~\eqref{th:joint:step3}.

{\it Step 2}. (Proof of~\eqref{item:cond:2}).
Since $X_0$ is distributed as the stationary distribution $\mu$ then 
$ 
	\sup_{u\in [0,\infty)} \EE\!\left[  |X_u|^m \right] 
	= \EE\!\left[ |X_0|^m\right] = \int_{-\infty}^{\infty} |x|^m \mu(\!\vd x)
<\infty.
$
This, the tower property, and \eqref{th:joint:step3} 
imply that there exists $C\in (0,\infty)$ depending only on $m,j, a_\pm, b_\pm, \sigma_\pm$ such that
\[\begin{split}
	&\frac{1}{T_N^{\nicefrac1\lambda}}\int_{0}^{T_N} \EE\left[ |X_t - X_{\lfloor t \rfloor_{\Delta_N}}|^{j} |X_{\lfloor t \rfloor_{\Delta_N}}|^{m}  \right] \!\vd t 
\\
	& \leq C \sqrt{\Delta_N^j T_N^{2(1-\lambda^{-1})}} = C \sqrt{\Delta_N^j T_N^{(\lambda-1)}} \xrightarrow[N\to\infty]{} 0
\end{split}\]

{\it Step 3}.  (Proof of~\eqref{item:cond:1}).
Let $s,t\in [0,\infty)$ be fixed such that $t-s\in [0,\Delta_{N}]$. 
Let us first note that we just need to consider $\EE\left[   \ind{\{ \pm X_t<0\}} \ind{\{\pm X_{s} > 0\}}  |  X_{s} \right]$.
This, given $X_s$, is bounded by
$ \PP\left( \tau_{s,\pm} \leq t-s \right) \leq \PP\left( \tau_{s,\pm} \leq \Delta_N \right) $ where $\tau_s$ is the first hitting time of the level 0 of the OU process solution to the following SDE: 
$
	\xi_u = X_s + \int_0^{u} b_\pm -  a_\pm \xi_v \vd v +  \sigma_\pm  W_u^s  
$ 
with $W^s$ a Brownian motion independent of $\sigma(X_v, v\in [0,s])$.
If $a_\pm \neq 0$, \cite[Section 6.2.1]{Lipton_Kaushansky_2018} with 
	$b =-\frac{b_\pm}{\sqrt{a_\pm} \sigma_\pm}$, 
	$z= \frac{\sqrt{a_\pm}}{\sigma_\pm} X_s -\frac{b_\pm}{\sqrt{a_\pm} \sigma_\pm}$, 
	and $t= {a_\pm} \Delta_N$,
prove that
\[
	\PP\left( \tau_{s,\pm} \leq \Delta_N \right)  = 2 e^{- \frac{b_\pm}{\sigma_\pm^2} X_s} \Phi\!\left(- \frac{\sqrt{a_\pm}}{\sigma_\pm} |X_s| \gamma_N\right)
\text{ with } \gamma_N:= \frac{e^{-\frac{a_\pm \Delta_N}{2} }}{\sqrt{\sinh(a_\pm \Delta_N)}}.
\]
If $a_\pm =0$ and $\pm b_\pm<0$ then 
\[
\begin{split}
	\PP\left( \tau_{s,\pm} \leq \Delta_N \right)  
&= 
\int_0^{\Delta_N} \frac{|X_s|}{\sigma_\pm\sqrt{2\pi u^3}} 
\exp\left(
-\frac{(X_s-b_\pm u)^2}{2\sigma_\pm^2 u}
\right)	
\vd u	
\\
& \leq \left(1+e^{\frac{2 |X_s| |b_\pm|}{\sigma_\pm^2}}\right) \Phi\!\left(- \frac{|X_s| - |b_\pm| \Delta_N}{ \sqrt{2 \sigma_\pm^2 \Delta_N } }\right) .
\end{split}
\]
Therefore, using the stationary distribution~\eqref{eq:stat:dist}, to establish Step~3 it suffices to prove that the following quantity vanishes as $N\to\infty$:
\[
\frac{1}{T_N^{\nicefrac1\lambda}} \int_0^{T_N} \EE\!\left[ |X_{\lfloor t \rfloor_{\Delta_N}}|^m \PP\left( \tau_{\lfloor t \rfloor_{\Delta_N},\pm} \leq \Delta_N \right) 
\right]  
 \! \vd t.
\]
Let us first consider the case
$a_\pm =0$ and $\pm b_\pm <0$.
The desired quantity is bounded by
\begin{footnotesize}
\[ 
\begin{split}
	& \frac{T_N}{T_N^{\nicefrac1\lambda}}  \frac{1}{\int_{-\infty}^{\infty} m(y) \vd y} \int_{\RR^\pm} \frac{2 |y|^m}{(\sigma_\pm)^2} \exp\!\left( \frac{2  y  b_\pm }{(\sigma_\pm)^2} \right)
 \left(1+e^{\frac{2 |y| |b_\pm|}{\sigma_\pm^2}}\right) \Phi\!\left(- \frac{|y| - |b_\pm| \Delta_N}{ \sqrt{2 \sigma_\pm^2 \Delta_N} }\right) \! \vd y
	\\
	& \leq C_1 \frac{T_N}{T_N^{\nicefrac1\lambda}}  \frac{1}{\int_{-\infty}^{\infty} m(y) \vd y} \int_{\RR^\pm}  \exp\!\left(- C_2  |y|^2/\Delta_N \right) \! \vd y
\leq C_3 T_N^{1-\frac1\lambda} \sqrt{\Delta_N} 
\xrightarrow[N\to\infty]{} 0
\end{split}
\]
\end{footnotesize}
for constants $C_1,C_2,C_3\in (0,\infty)$ depending on $a_\pm,b_\pm,\sigma_\pm$.
Let us now consider the case $a_\pm>0$ and $b_\pm \in \RR$. 
The desired quantity is bounded by
\begin{small}
\[ 
\begin{split}
&	\frac{T_N}{T_N^{\nicefrac1\lambda}}  \frac{1}{\int_{-\infty}^{\infty} m(y) \vd y} \int_{\RR^\pm} \frac{2 |y|^m}{(\sigma_\pm)^2} \exp\!\left( - \frac{ y (a_\pm y- 2 b_\pm) }{(\sigma_\pm)^2} \right)
  2 e^{- \frac{b_\pm}{\sigma_\pm^2} y} \Phi\!\left(- \frac{\sqrt{a_\pm}}{\sigma_\pm} |y| \gamma_N\right) \! \vd y
	\\& 
\leq C_1 \frac{T_N}{T_N^{\nicefrac1\lambda}}  \frac{1}{\int_{-\infty}^{\infty} m(y) \vd y} \int_{\RR^\pm}  \exp\!\left(- C_2  |y|^2 \gamma_N^2\right) \! \vd y
\leq C_3 \frac{T_N^{1-\frac1\lambda}}{\gamma_N} 
\end{split}
\]
\end{small}
for constants $C_1,C_2,C_3\in (0,\infty)$ depending on $a_\pm,b_\pm,\sigma_\pm$.
The latter term vanishes since 
$ \lim_{N\to\infty} \frac{T_N^{1-\frac1\lambda}}{\gamma_N} \leq \lim_{N\to\infty}T_N^{1-\frac1\lambda}\sqrt{\Delta_{N}}=0.
$
The proof is thus completed.
\end{proof}

\begin{proof}[Proof of Item~\eqref{th2:item:LAN} of Theorem~\ref{th:joint:CLT}]
Similarly to Item~\eqref{th1:item:LAN} of Theorem~\ref{th:continuous}.
The analogous of $A^\pm_T$ is
$
	A_{T_N}^{N,\pm}:= \begin{pmatrix}
	-\cM_{T_N,N}^{\pm,1} 
	-
	\frak a_\pm \cQ^{\pm,2}_{T_N,N}
	+ \frak b_\pm \cQ^{\pm,1}_{T_N,N}
, &
	\cM_{T_N,N}^{\pm,0} 
	-
	\frak b_\pm \cQ_{T_N,N}^{\pm,0} 
	+ \frak a_\pm \cQ_{T_N,N}^{\pm,1} 
	\end{pmatrix}
$
and Lemma~\ref{lemma:control:diff:disc:cont} ensures that the asymptotic behavior of the latter quantity is the same as the one of the continuous time analogue, Theorem~\ref{th:joint:CLT}.\eqref{th1:item:LAN}.
\end{proof}

\subsection{Proof of Theorem~\ref{th:disc}}\label{secbig:proofth:discrete:LLNCLT}

The proof of Item~\eqref{th3:item:QMLE} of Theorem~\ref{th:disc} is along the lines of the one of Theorem~\ref{th:joint:CLT}.\eqref{th2:item:QMLE}.
\\
The proof of Item~\eqref{th3:item:LLN} of Theorem~\ref{th:disc} is based on Lemma~\ref{lemma:conv:M} and  Lemma~\ref{lem:conv:occ:time} below.
\\
Let us be more precise. 
For all $T\in (0,\infty)$ and $N\in \NN$ the difference
$
\begin{pmatrix}
 \widehat{a}^\pm_{T,N}
-\alpha_{T}^\pm, 
&
 \widehat{b}^\pm_{T,N}
-\beta_{T}^\pm
\end{pmatrix}
$
can be rewritten, using Theorem~\ref{th:continuous}.\eqref{th1:item:QMLE} and Theorem~\ref{th:disc}.\eqref{th3:item:QMLE}, as an expression involving only terms of the kind
\[
\left(\frac{\cQ^{\pm,i}_{T,N}}{ \cQ^{\pm,0}_{T,N}\cQ^{\pm,2}_{T,N}	-(\cQ^{\pm,1}_{T,N})^2} 
- \frac{\cQ^{\pm,i}_{T}}{ \cQ^{\pm,0}_{T}\cQ^{\pm,2}_{T}	-(\cQ^{\pm,1}_{T})^2}  \right) \cM^{\pm,j}_{T}
+ \frac{\cQ^{\pm,i}_{T,N}  { ( \cM^{\pm,j}_{T,N} - \cM^{\pm,j}_{T} )} }{ \cQ^{\pm,0}_{T,N}\cQ^{\pm,2}_{T,N}	-(\cQ^{\pm,1}_{T,N})^2}
\] 
for $j \in \{0,1\}$, $i\in \{0,1,2\}$.
The convergence 
$(
    \widehat{a}^+_{T,N},\widehat{b}^+_{T,N}
    ,\widehat{a}^-_{T,N},\widehat{b}^-_{T,N} )
\convp[N\to\infty]
   ( \alpha^+_{T},\beta^+_{T}
    ,\alpha^-_{T},\beta^-_{T} ) 
$
is obtained combining Lemma~\ref{eq:Qinfinity2} and 
Theorem~2.2 in \citep{crimaldi}  
with the convergences in probability 
in Lemma~\ref{lemma:conv:M} and Lemma~\ref{lem:conv:occ:time} below.
The proof of~\eqref{eq:CLT:T} relies on Lemma~\ref{eq:Qinfinity2}, 
Theorem~2.2 in \citep{crimaldi}  
Lemma \ref{lem:conv:occ:time} and equation \eqref{lemma:conv:M0} in Lemma~\ref{lemma:conv:M}.

\begin{lemma}\label{lemma:conv:M}
Let $m=0,1$.
Then
$
    \cM^{\pm,m}_{T,N}
     \convp[{N\to\infty}]
    \cM^{\pm,m}_{T}
$
and
\begin{equation} \label{lemma:conv:M0}
  N^{1/4} ( \cM^{\pm,m}_{T,N}
    -\cM^{\pm,m}_{T}) \xrightarrow[N\to \infty]{\mathrm{stably}} \pm {r^m} \sqrt{\frac{4 \sqrt{T} }{3 \sqrt{2 \pi}}  \frac{\sigma_-^2 +\sigma_+^2}{\sigma_-+\sigma_+}} B_{L^r_T(X)}
\end{equation}
where $B$ is a Brownian motion independent of $X$.
\end{lemma}

\begin{lemma} \label{lem:conv:occ:time}
Let $m\in \NN$.
Then
$
	\sqrt{N}(    \cQ^{\pm,m}_{T,N} - \cQ^{\pm,m}_{T})  
	\convp[N\to\infty] 0.
$\\
Moreover, if the threshold $r=0$, then 
$
	{\sqrt{N^{1+m\varepsilon}}} (    \cQ^{\pm,m}_{T,N} - \cQ^{\pm,m}_{T}) 
	\convp[N\to\infty] 0
$
for every $\varepsilon <1$. 
\end{lemma}
The latter convergence result extends~\cite[Theorem~4.14]{LP}, where only $m=0$ is considered. The proof strategies are the same. 

In the remainder of the section we prove Lemma~\ref{lemma:conv:M} and Lemma~\ref{lem:conv:occ:time} below
for $X$ solution to~\eqref{eq:AffDOBM}.

\subsubsection{Proof of Lemma~\ref{lemma:conv:M}}
Without loss of generality we can assume $X_0$ deterministic
and also 
reduce 
ourselves 
to prove all results of the section in the case of null drift, i.e.~$X$ is an oscillating Brownian motion (OBM). 
Indeed all statements are about convergence in probability or stable convergence, and, once these convergences have been proved for the null drift case, they can be extended to the drifted case (piecewise linear drift) 
using the fact that 
Girsanov weight is an exponential martingale
and dominated convergence theorem. In the case of convergence in probability 
one proves that for every sub-sequence there exists a sub-sub-sequence converging a.s., instead stable convergence follows by property~\eqref{scl} and Skorokhod representation theorem.

Therefore, in the remainder of the section, let $X$ be an OBM with deterministic starting point $X_0$ and let $T\in (0,\infty)$ be fixed.

\begin{lemma}\label{lemma:1}
It holds that
\begin{equation*} 
\begin{split}
	N^{\frac14} \bigg( \sum_{k=0}^{N-1} (X_{(k+1)T/N}-X_{k T/N})^2 \ind{\{\pm (X_{k T/N}- r)> 0\}}- \sigma_{\pm}^2 \int_0^T \ind{\{\pm (X_s -r)>0\}} \vd s \bigg)
\convp[N\to\infty]
	0.
\end{split}
\end{equation*}
\end{lemma}
\begin{proof}[Proof of Lemma~\ref{lemma:1}]
We write $X^{r,\pm}:= (X-r) \ind{\{\pm (X - r) >0\}}$, $X_k:=X_{k \frac{T}{N}}$,
and $\Delta_k X:=X_{k+1}-X_k$.
Moreover we can assume 
$r=0$ 
(just note that given $X$ with threshold $0$, $\eta=X+ r$ has threshold $r$ and $\Delta_i X=\Delta_i \eta$, $\Delta_i X^{0,\pm}=\Delta_i \eta^{r,\pm}$).

We observe that
\[
\begin{split}
	&
\sum_{i=0}^{N-1} (\Delta_i X)^2 \ind{\{ \pm X_i >0\}} 
	 =
\sum_{i=0}^{N-1} \Delta_i X \Delta_i X^{0,\pm} 
\mp
\sum_{i=0}^{N-1} (\Delta_i X) |X_{i+1}|
\ind{\{X_i X_{i+1} < 0\}}.
\end{split}
\]
Proposition~2 in \cite{mazzonetto2019rates} 
(or \cite[Proposition~2]{LMT1}) 
ensures that
\[\begin{split}
& \frac{1}{N^{1/4}} \cdot N^{1/2}
\sum_{i=0}^{N-1} (\Delta_i X) |X_{i+1}| \ind{\{X_{i}X_{i+1}<0\}} 
\convp[N\to\infty] 
0 \cdot \frac{2 \sqrt{2} (\sigma_+^2 + \sigma_-^2)}{ 3 \sqrt{\pi} (\sigma_+ + \sigma_-) } L_T^0(X) =0.
\end{split}
\]
%
Now consider the remaining term. In \citep{LP} (cf.~proof of Theorem~3.5, page 3594) it is shown that
there exists a constant $C\in \RR$ such that
\begin{equation*}
\begin{split}
	\sqrt{N}\bigg( \sum_{k=0}^{N-1} (\Delta_k X)^2 \ind{\{\pm X_k> 0\}}- \sigma_{\pm}^2 \int_0^T \ind{\{\pm X_s >0\}} \vd s \bigg)
\xrightarrow[N\to \infty]{\mathrm{stably}}
	& \\
	\quad \sqrt{2 T} \sigma_\pm^2 \int_0^T \ind{\{\pm X_s > 0\}}  d {B}_s
\mp C L^0_T(X) &
\end{split}
\end{equation*}
where $ {B}$ is a Brownian motion independent of 
$X$. 
%
%
Using
~\eqref{scl} completes the proof. 
\end{proof}

%
%
We are now ready to prove Lemma~\ref{lemma:conv:M}.
\begin{proof}[Proof of Lemma~\ref{lemma:conv:M}]  
Let $\{\cdot\}^\pm$ denote positive and negative part.
Note that for all $t\in [0,T]$ it holds $\PP$-a.s.~that $  (X_t-r) \ind{\{\pm (X_t - r) >0\}} = \pm \{ X_t - r \}^{\pm}$.
Applying 
It\^o-Tanaka formula
establishes that the following equalities hold $\mathbb{P}$-a.s.:
\begin{equation} \label{conv:expr:2}
\begin{split}
	(X_T +r)^i \{X_T -r\}^\pm - (X_0 +r)^i \{X_0 -r\}^\pm 
	& =
		\pm 2^i \cM^{\pm,i}_T
		\pm i \sigma_\pm^2 \cQ_T^{\pm,0}
+ \frac{r^i}{2^{1-i}} L^r_T(X),
		\end{split}
\end{equation}
%
%
for $i=0,1$. Next note that it holds $\PP$-a.s.~for all $i\in \{0,\ldots,N-1\}$ that
\begin{equation*}
   \ind{\{\pm (X_i - r) > 0\}} \Delta_i X = \pm \{X_{i+1}-r\}^\pm \mp \{X_i-r\}^\pm \mp \ind{\{(X_i-r)(X_{i+1}-r)<0\}}|X_{i+1}-r|
\end{equation*}
and
\begin{equation*}
\begin{split}
    &
    2 X_i \ind{\{\pm (X_i - r) > 0\}} \Delta_i X
    \\
    &=
   (X_{i+1}^2-r^2) \ind{\{\pm (X_{i+1}- r) > 0\}} - (X_{i}^2-r^2)  \ind{\{\pm (X_i - r) > 0\}}
   - (X_{i+1}-X_i)^2  \ind{\{\pm (X_i - r) > 0\}}
   \\
   & \quad 
   \mp 2 r \ind{\{(X_i-r)(X_{i+1}-r)<0\}} |X_{i+1}-r| \mp \ind{\{(X_i-r)(X_{i+1}-r)<0\}} |X_{i+1}-r| (X_{i+1}-r) .
   \end{split}
\end{equation*}
Combining this with \eqref{def:LT:discr} and \eqref{conv:expr:2}
imply that 
it holds $\PP$-a.s.~that
\begin{equation} \label{eq:M:local:equal}
	\mp 2 \left( \cM^{\pm,0}_{T,N}-\cM^{\pm,0}_{T} \right) =  L^r_{T,N} - L^r_T (X) \qquad \text{and}
\end{equation}
\begin{equation*}
\begin{split}
	 2 \left(\cM^{\pm,1}_{T,N}-\cM^{\pm,1}_{T} \right) 
	& =  \mp r (L^r_{T,N} - L^r_T(X)) -  \sum_{k=0}^{N-1} (X_{k+1}-X_k)^2  \ind{\{\pm (X_k - r) > 0\}}  +  \sigma_\pm^2 \cQ^{\pm,0}_T 
	\\
	& \quad \mp \sum_{k=0}^{N-1} (X_{k+1}-r) |X_{k+1}-r|  \ind{\{(X_k-r)(X_{k+1}-r)<0\}}.
\end{split}
\end{equation*}
The result 
follows from 
\cite[Proposition~7]{mazzonetto2019rates}, 
Lemma~\ref{lemma:1}, and from the following convergence:
\[
\sqrt{N} \sum_{i=0}^{N-1} (X_{i+1}-r)  |X_{i+1}-r| \ind{\{(X_{i}-r)(X_{(i+1)}-r)<0\}}  
\convp[N\to\infty] \frac{2 \sqrt{2}}{3 \sqrt{\pi}} (\sigma_+-\sigma_-) L^r_T(X)
\]
which follows from \cite[Proposition~2]{mazzonetto2019rates} or \cite[Proposition~2]{LMT1}.
\end{proof}

\subsubsection{Proof of Lemma~\ref{lem:conv:occ:time}} \label{sec:conv:occ:time}

As in the previous section, let $X$ be an OBM with deterministic starting point $X_0$ and let $T\in (0,\infty)$ be fixed. 
We reduce to consider threshold $r=0$ because for $t\in (0,\infty)$ it holds that 
	$
	\cQ^{\pm,m}_{t} := \cQ^{\pm,m}_{t}(X,r) = \sum_{k=0}^m \binom{m}{k} r^{m-k}\cQ^{\pm,k}_{t}(X-r,0)
	$
and the same holds for $\cQ^{\pm,m}_{N,t}$.

The following result follows from Lemma~4.3 in \citep{LP} and the scaling property for OBM.
\begin{lemma}
\label{convY}
Let $f$ be a bounded function such that $\int |x|^k |f(x)|\vd x<\infty$ for $k=0,1,2$.
Then for all $T\in (0,\infty)$
\begin{equation} \label{conv:fm}
	 (N/T)^{-1/2} \sum_{i=0}^{N-1} f(\sqrt{N/T} \, X_{i T/N}) \convp[N\to\infty] \lambda_\sigma(f) L_T^0(X)
\end{equation}
where 
$
\lambda_\sigma(f):=\left(\frac1{\sigma_+^2}  {\int_{0}^{\infty}  f(x) \vd x} +
\frac1{\sigma_-^2} {\int_{-\infty}^{0}  f(x) \vd x}\right)
$.
\end{lemma}

Let us denote by $\cG$ the natural filtration associated to the process $X$ (or equivalently to its driving BM).
Let $m\in \NN$ be fixed.
For $i=1,\dotsc ,N$, we consider
$X_{i-1,N}:= X_{(i-1)T/N}$,
$\mathcal{G}_{i-1,N}:=\mathcal{G}_{(i-1)T/N}$,
and
\begin{equation}\label{setJUN}
\begin{split}
J_{i,N}^{(m)}&= \left(\frac{T}{N} X_{i-1,N}^m \ind{\{\pm X_{i-1,N}\geq 0\}}
-
 \int_{\frac{(i-1)T}{N}}^{\frac{i T}{N}} X_s^m \ind{\{\pm X_s\geq 0\}} \vd s \right)\\
 &=
 \pm \sgn(X_{i-1,N}) X_{i-1,N}^m 
\int_{\frac{(i-1)T}{N}}^{\frac{i T}{N}} \ind{\{X_{i-1,N} X_s<0\}} \vd s
+
\int_{\frac{(i-1)T}{N}}^{\frac{i T}{N}} (X_{i-1,N}^m-X_s^m) \ind{\{
 \pm X_s > 0\}} \vd s, \\
U_{i,N}^{(m)}&=J_{i,N}^{(m)}-\EE[{J_{i,N}^{(m)}| \cG_{i-1,N}} ].
 \end{split}
 \end{equation}
Observe that $U_{i,N}^{(m)}$ are martingale increments and  
\begin{equation*}
    \cQ^{\pm,m}_{T,N} - \cQ^{\pm,m}_{T}
=
\sum_{i=1}^N \EE[ J_{i,N}^{(m)}  | \cG_{i-1,N} ]
+
\sum_{i=1}^N U_{i,N}^{(m)}.
\end{equation*}
The following lemma proves the convergence of the two terms. 
%
%
%
\begin{lemma} \label{lem:Jm:T_N}
Let $\varepsilon\in [0,1)$, $m\in \NN$,
and
let $\cG$, $J_{i,N}^{(m)}$ and $U_{i,N}^{(m)}$, $i\in \{1,\ldots,N\}$ defined by \eqref{setJUN}.
Then 
\begin{enumerate}[i)]
\item \label{item:1:Jm}
$
N^{\frac{1+m\varepsilon}{2}} 
\sum_{i=1}^N \EE[{ J_{i,N}^{(m)}  | \cG_{i-1,N} }]
\convp[N\to\infty] 0
$
and
\item \label{item:2:Um}
$
N^{\frac{1+m\varepsilon}{2}} 
 \sum_{i=1}^N  U_{i,N}^{(m)} \convp[N\to\infty]0.
$
\end{enumerate}
\end{lemma}

\begin{proof}[Proof of Lemma~\ref{lem:Jm:T_N}]
In this proof we use the following notation:
For every $q\in [0,\infty)$ let $f_m$, $g_{m,q}$, $h_q$ be the real functions satisfying
\[
	f_m(x) = 
	\begin{cases}
		\frac{ 2\sigma_+}{\sigma_-+\sigma_+} \int_0^{1}  x^m  \Phi(-x/(\sigma_+\sqrt{t})) \vd t & \text{ if } x \geq 0
		\\
		\frac{ - 2\sigma_-}{\sigma_-+\sigma_+} \int_0^{1}  x^m  \Phi(x/(\sigma_-\sqrt{t})) \vd t & \text{ if } x<0
	\end{cases}
\] 
with $\Phi = \frac{1}{\sqrt{2\pi}} \int_{-\infty}^\cdot e^{-\frac{y^2}2} \vd y$, 
$h_q(x)= |x|^q|f_0(x)| $,
and
\[
	g_{m,q}(y) :=
	\begin{cases}
	 \int_{0}^{1} \int_0^\infty |y^m-x^m|^q \frac{1}{\sqrt{2 \pi t}} \frac{1}{\sigma_+} e^{-\frac{1}{2 t} \left(\frac{x}{\sigma_+}- \frac{y}{\sigma_+}\right)^2}  
	\left( 1 + \frac{\sigma_--\sigma_+}{\sigma_- +\sigma_+} e^{-\frac{4 x y}{2 t \sigma_+^2} } 
	\right) 
	\! \vd x \vd t & \text{ if } y\geq 0
	\\
	 \int_{0}^{1} \int_0^\infty |y^m-x^m|^q \frac{1}{\sqrt{2 \pi t}} 	\frac{2 \sigma_-}{\sigma_-+\sigma_+} e^{-\frac{1}{2 t} \left(\frac{x}{\sigma_+}- \frac{y}{\sigma_-}\right)^2}	\! \vd x \vd t & \text{ if } y<0.
	\end{cases}
\]
Let us show that the functions above satisfy the assumptions of Lemma~\ref{convY}.
Indeed, since $\Phi(-x) \ind{\{x \geq 0\}} \leq \frac12 e^{-{x^2}/2}$, it holds
\[
	|f_0(x)| \leq 
	\frac{\sigma(x)}{\sigma_-+\sigma_+} \int_0^1 e^{-\frac{x^2}{2 t \sigma(x)^2}} \vd t
	\leq \frac{\sigma(x)}{\sigma_-+\sigma_+} e^{-\frac{x^2}{2 \sigma(x)^2 }}.
\]
This ensures that the coefficients (defined in \eqref{conv:fm}) 
$\lambda_\sigma(f_m)$, $\lambda_\sigma(h_q)$ are finite.
Moreover it can be shown that $\lambda_\sigma(g_{m,q})<\infty$ for $q\in [0,\infty)$.
In particular note that
\[
\lambda_\sigma(f_m)=
\frac{2(\sigma_+^{m}-(-\sigma_-)^{m})}{\sigma_-+\sigma_+} \int_0^\infty
x^m  \int_0^1  
 \Phi(-x/\sqrt{s})
  \vd s \vd x.
\]
Hence Lemma~\ref{convY}
shows for all $q\in [0,\infty)$, $t\in (0,\infty)$, $f\in \{f_m,h_q,g_{m,q}\}$ that 
\begin{equation} \label{conv:fm:0}
	 \frac{1}{\sqrt {N}} \sum_{i=0}^{N-1} f(\sqrt{N/T} \, X_{i T/N}) \convl[N\to\infty] \frac{\lambda_\sigma(f)  }{\sqrt {T}}  L_T(X).
\end{equation}

Let us first show an easy useful equality, which uses the explicit expression of the transition density of the OBM
\begin{equation} \label{eq:obm:density}
	q_{\sigma}(t,\sigma(x)x,\sigma(y)y) = \frac{1}{\sqrt{2\pi t} \sigma(y)} \left( e^{-\frac{(x-y)^2}{2t}}
  +\frac{(\sigma_--\sigma_+)}{(\sigma_-+\sigma_+)} \sgn(y) e^{-\frac{(|{x}|+|{y}|)^2}{2t}}\right).
\end{equation}
Let $Y$ be an OBM with starting point $Y_0$ and threshold $r=0$, then for all $c \in \{-1,+1\}$, $q\in [0,\infty)$, Fubini, 
~\eqref{eq:obm:density},
a change of variable yield
\begin{equation} \label{eq:functionf:eq}
\begin{split}
	& \ind{\{c Y_0 >0\}} \EE\!\left[ \int_{0}^{\frac{T}{N}}  |Y_{0}|^{q} \ind{\{c Y_s<0\}} \vd s | Y_{0} \right]
	=
	\ind{\{c Y_0 >0\}}  \int_{0}^{\frac{T}{N}}  |Y_{0}|^{q}  \EE\!\left[  \ind{\{c Y_s<0\}} | Y_{0} \right]  \vd s
	\\
	&
 	=
 	\ind{\{c Y_0 >0\}}   (T/N)^{\frac{q}{2}+1} 
	\int_0^{1}  
 	\frac{2\sigma(Y_0) |\sqrt{N/T}\,Y_0|^{q}}{\sigma_-+\sigma_+}\Phi(- |\sqrt{N/T}\, Y_0|/(\sigma(Y_0)\sqrt{t})) 	\vd t.
\end{split}
\end{equation}

Let us now prove Item~\eqref{item:1:Jm}.\\
{\it First step.} Let $i\in \{1,\ldots,N\}$ be fixed.\\
We prove in this step that
\begin{equation}  \label{lm:cot:step1}
	N^{\frac{1+m\varepsilon}{2}} \sum_{k=0}^{N-1} \EE\!\left[ \int_{\frac{k T}{N}}^{\frac{(k+1) T}{N}} \pm \sgn(X_{k,N}) X_{k,N}^m \ind{\{X_{k,N}X_s<0\}} \vd s | \cG_{k,N}   \right] 	
	\convp[N\to\infty] 0.
\end{equation}

Note that the Markov's property
and 
\eqref{eq:functionf:eq} 
ensure that
\begin{equation} \label{eq:J1:OBM}
\begin{split}
	& 
\sqrt{N^{1+m \varepsilon}} \sum_{i=1}^N \EE\!\left[ \int_{\frac{(i-1)T}{N}}^{\frac{i T}{N}} \pm \sgn(X_{(i-1)T/N}) X_{(i-1)T/N}^m \ind{\{X_{(i-1)T/N}X_s<0\}} \vd s | \cG_{i-1,N}   \right]
 	\\
	& = 
	\pm  T^{1+\frac{m}{2}} \sum_{i=0}^{N-1}  N^{-\frac{1+m (1-\varepsilon)}{2}} f_{m}(\sqrt{N/T}X_{i T/N}).  
\end{split}
\end{equation}
This vanishes because equation~\eqref{conv:fm:0} holds and $\lambda_\sigma(f_0)=0$ and when $m\neq 0$ it holds $\varepsilon<1$.
The proof of \eqref{lm:cot:step1} is thus completed.

{\it Second step.}
Let $j\in \{1,2\}$.
We prove now that
\begin{equation} \label{eq:conv:t2}
N^{\frac{j(1+m \varepsilon)}{2}} 
\sum_{i=1}^N 
\EE\!\left[\int_{\frac{(i-1)T}{N}}^{\frac{i T}{N}} (X_{i-1,N}^m-X_s^m)^j \ind{\{
 X_s > 0\}} \vd s
 | \cG_{i-1,N}   \right]
\convp[N\to\infty] 0.
\end{equation}

By the Markov property, a simple change of variable, Fubini, and the  explicit expression of the transition density of the OBM in~\eqref{eq:obm:density} 
we obtain for all $i\in \{1,\ldots, N\}$
\[\begin{split}
	&  
	 \EE\!\left[ 
	 \int_{\frac{(i-1)T}{N}}^{\frac{i T}{N}} |X_{\frac{(i-1)T}{N}}^m-X_s^m|^{j} \ind{\{
	 X_s > 0\}} \vd s | \cG_{i-1,N} \right]
	\\
	& = 
	\frac{T}{N}  
	\int_{0}^{1} 	\EE\!\left[ |X_{\frac{(i-1)T}{N}}^m-X_{t \frac{T}{N}}^m|^{j} \ind{\{
	 X_{t \frac{T}{N}} > 0\}}  | X_{\frac{(i-1)T}{N}} \right] 
	\vd t
= 
	 \left(\frac{T}{N}\right)^{\! \frac{m j}2 +1} g_{m, j}(\sqrt{{N}/{T}} \, X_{\frac{(i-1)T}{N}}).
\end{split}\]
Combining \eqref{conv:fm:0} with the fact that 
$\lambda_\sigma(g_{0,j})=0$ and $\varepsilon<1$
it follows that the latter quantity converges in probability to $0$ with the speed which proves~\eqref{eq:conv:t2}.
%
%
%
Taking $j=1$ establishes Item~\eqref{item:1:Jm}.

\emph{Third step}. (Proof of Item~\eqref{item:2:Um}).
Note that Jensen's inequality implies that
\begin{equation*}
\begin{split}
&  \EE[{ (U_{i,N}^{(m)})^2|\cG_{i-1,N} }] 
=  \EE[{  (J_{i,N}^{(m)})^2|\cG_{i-1,N} }] - \left( \EE[{  J_{i,N}^{(m)}|\cG_{i-1,N} }] \right)^2 \leq \EE[{ (J_{i,N}^{(m)})^2|\cG_{i-1,N} }]
\\
& \leq
 \EE\!\left[ \frac{2 T}{N}  X_{i-1,N}^{2m}
\int_{\frac{(i-1)T}{N}}^{\frac{i T}{N}} \ind{\{X_{i-1,N}
 X_s<0\}}  \vd s
+\frac{2 T}{N}
\int_{\frac{(i-1)T}{N}}^{\frac{i T}{N}} ( X_{i-1,N}^m-X_s^m)^2 \ind{\{
 X_s > 0\}}  \vd s  |\cG_{i-1,N} \right]
\! .
\end{split}
\end{equation*}
%
%
%
This and~\eqref{eq:conv:t2} with $j=2$ ensure that
it suffices to prove
\[
	N^{\frac1{2}+m\varepsilon}   \EE\!\left[ \frac{2 T}{N}  X_{i-1,N}^{2m}
\int_{\frac{(i-1)T}{N}}^{\frac{i T}{N}} \ind{\{X_{i-1,N}
 X_s<0\}}  \vd s  |\cG_{i-1,N} \right]
\convp[N\to\infty] 0.
\]
By the Markov's property and \eqref{eq:functionf:eq}
we reduce to study the convergence of
\[
	2 N^{-\frac1{2}+ m (\varepsilon-1)} 
	\sum_{k=0}^{N-1} 
	  N^{-\frac12} h_{2 m}( \sqrt{N/T} \,X_{k T/N}).
\]
It follows from~\eqref{conv:fm:0} that the latter quantity converges to 0 in probability as $N\to\infty$.
We have therefore obtained that 
\[N^{1+m\varepsilon}  \sum_{i=1}^N \EE[{ (U_{i,N}^{(m)})^2|\cG_{i-1,N} }] \convp[N\to\infty] 0.\]
Applying Theorem~4.4 in \citep{LP} completes the proof.
\end{proof}

\section{The multi-threshold Ornstein-Uhlenbeck process} \label{sec:multit}

Let us consider in this section the multi-threshold version of the threshold OU process, by which we mean the solution to
\begin{equation}
    \label{eq:AffDOBM:multi}
    X_t=X_0+\int_0^t \sigma(X_s)\vd W_s+\int_0^t \left( b(X_s) - a(X_s) \, X_s \right) \vd s , \quad t\geq 0,
\end{equation}
with
piecewise constant coefficients $\sigma$, $a$, and $b$ possibly discontinuous at levels $-\infty=r_0 < r_1 < \ldots < r_d< r_{d+1}=+\infty$, $d \in \NN$.
Let $I_0=(-\infty,r_1)$ and $I_j:=[r_j,r_{j+1})$, for $j\in \{1,\ldots,d\}$. 
The volatility coefficient is given by
	\begin{equation}
    \label{sigmaDOBM:multi}
    \sigma(x)= \sum_{j=0}^{d} {\sigma}_{j} \ind{I_j}(x)  
>0
\end{equation}
and similarly the drift coefficients are given by
\begin{equation}\label{SETvas:multi}
\begin{split}
    & b(x)= \sum_{j=0}^d { b}_{j} \ind{I_j}(x) 
	\quad \text{and}\quad a(x)= \sum_{j=0}^d { a}_{j} \ind{I_j}(x). \end{split}
\end{equation} 
In analogy to the result for $d=1$ we also denote $a_0, b_0, \sigma_0$ by $a_-,b_-,\sigma_-$ and $a_d, b_d, \sigma_d$ by $a_+,b_+,\sigma_+$.

\subsection{The regimes of the process} \label{sec:regimes}

In this section, we establish for which values of the coefficients the process $X$ is (positively or null) recurrent or transient.
Recall that we denote {\it scale function} and {\it speed measure} respectively by $S$ and $m$. The derivative (up to a multiplicative constant) of the scale function satisfies 
for all $j=0,1$ and $k=2,\ldots,d$
\[
	S'(x) \ind{I_j}(x) =  \frac1{s_j(x,r_1)} \quad \text{and} \quad S'(x) \ind{I_k}(x) =  \frac1{s_k(x,r_k) \prod_{i=1}^{k-1} s_i(r_{i+1},r_i)}
\]
where for all $k\in \{0,1,\ldots,d\}$, $x,r\in \RR$ we define the functions 
\[
	s_k(x,r):=  \exp{\!\left(  \frac{2 b_k (x-r) - a_k(x^2-r^2)}{\sigma_k^2}  \right)}.
\]
The speed measure is $m(x) \vd x = \frac{2}{\sigma(x)^2 S'(x)} \vd x$.

Recall that {$X$ is recurrent} if and only if $\lim_{x\to +\infty} S(x)=+\infty$ and $\lim_{x\to -\infty} S(x)=-\infty$, which happens if and only if
{
[($a_+>0$ and $b_+\in \RR$) or ($a_+=0$ and  $b_+\leq 0$)] and
[($a_->0$ and $b_-\in \RR$) or ($a_-=0$ and  $b_-\geq 0$)]}.
The complementary leads to transience.
If $X$ is recurrent and the speed measure is a finite measure, then $X$ is  positive recurrent and ergodic. It admits a stationary distribution, denoted by $\mu$, which is the renormalized speed measure.
Hence we have the ergodicity condition~\eqref{eq:ergodic:cases} (the one for the single threshold case $d=1$).

\begin{lemma}[Multi-threshold version of Lemma~\ref{lem:ergodic}]
\label{lem:ergodic:multi}
The speed measure is finite if and only if condition~\eqref{eq:ergodic:cases} holds.
More precisely, let $C_0 = C_1=1$, $C_j = \prod_{k=1}^{j-1} s_k(r_{k+1},r_k), j=2,\ldots, d$, 
let
\begin{equation*}
	\mathfrak{m}_{i,j,k}:=  \frac{\sqrt{\pi}}{\sigma_i \sqrt{a_i}} 
	\exp{\!\left(\frac{a_i}{\sigma_i^2} \left(\frac{b_i}{a_i} -r_j \right)^{\! 2} \right)}
	\operatorname{erfc}{\!\left(-\frac{\sqrt{a_i}}{\sigma_i} \left(\frac{b_i}{a_i}-r_k\right)\right)}
\end{equation*}
with $i\in \{0,\ldots,d \}$ and $j,k\in \{1,\ldots, d\}$ and let $\mathfrak{m}_+ = - \mathfrak{m}_{0,1,1}$ and $\mathfrak{m}_-=\mathfrak{m}_{d,d,d}$.
Then if $j\in \{0,d\}$
\[
\int_{-\infty}^\infty \ind{I_j}(y) m(y) \vd y =
\begin{cases}
+\infty & \text{ if } a_\pm=0 \text{ and } b_\pm=0, \\
\frac{C_j }{|b_\pm|} & \text{ if } a_\pm=0 \text{ and } (-1)^\pm b_\pm > 0, \\
 C_j \mathfrak{m}_{\pm}& \text{ if } a_\pm>0 \text{ and } b_\pm \in \RR
\end{cases}
\]
and if $j\in \{1,\ldots, d-1\}$
\[
\int_{-\infty}^\infty \ind{I_j}(y) m(y) \vd y =
\begin{cases}
\frac{r_{j+1}-r_j}{\sigma_j^2} & \text{ if } a_j=0 \text{ and } b_j=0, \\
\frac{C_j}{b_j} \exp{\!\left(\frac{2 b_j (r_{j+1}-r_j)}{\sigma_j^2} -1 \right)}   & \text{ if } a_j=0 \text{ and }  b_j \neq 0, \\
C_j (\mathfrak{m}_{j,j,j}-\mathfrak{m}_{j,j,j+1})  & \text{ if } a_j>0 \text{ and } b_j \in \RR,\\
C_j \int_{r_j}^{r_{j+1}} s_{j}(x,r_j) \vd x <\infty& \text{ if } a_j<0 \text{ and } b_j \in \RR.\\
\end{cases}
\]
\end{lemma}

\subsection{MLE and QMLE from continuous time observations}
We assume in this section to observe the process on the time interval $[0,T]$, $T\in (0,\infty)$.
For $T\in (0,\infty)$, $m=0,1,2$, and $j =0,1,\ldots,d$ we define 
\begin{equation}\label{def:M:Q:multi}
 \cM_T^{j,m}:=
	 \int_0^T  X_s^m \ind{I_j}(X_s) \vd X_s
\quad \text{and} \quad	
\cQ_T^{j,m} 	
:=
\int_0^T X_s^m \ind{I_j}(X_s)  \vd s
\end{equation}
and take as likelihood function
$ G_T\dvec $ and as quasi-likelihood $\Lambda_T \dvec$ defined as in Section 2.

\begin{lemma}[Multi-threshold version of Lemma~2]  \label{eq:Qinfinity2:multi}
Assume the process is ergodic: condition~\eqref{eq:ergodic:cases} holds.
Then, for all $i\in \{0,1,2\}$, $j\in \{0,1,\ldots,d\}$, the quantities 
$\overline{\cQ}^{j,i}_\infty$ defined as follows are finite constants:
\[\overline{\cQ}^{j,i}_\infty \,{\overset{\as}{=}}\, \lim_{t\rightarrow \infty} \frac{{\cQ}^{j,i}_t}{t} = \int_{I_j} x^i \mu(\!\vd x)  \in \RR.\]
\end{lemma}

\begin{theorem} \label{th:continuous:multi}	
\begin{enumerate}[i)]
\item \label{th1:item:QMLE:multi}
For every $T\in (0,\infty)$ the MLE and QMLE are given by 
$\destim = \destimj_{j=0}^d$ with
\begin{equation} \label{eq:QMLE_ct:multi}	
\destimj
=
\begin{pmatrix}
	\frac{	\cM^{j,0}_T  \cQ^{j,1}_T			
	-\cQ^{j,0}_T
\cM^{j,1}_T
	}{ 
\cQ^{j,0}_T 
\cQ^{j,2}_T	-(\cQ^{j,1}_T)^2},
		&
	\frac{
\cM^{j,0}_T \cQ^{j,2}_T-\cQ^{j,1}_T	
\cM^{j,1}_T	}{
\cQ^{j,0}_T
\cQ^{j,2}_T-(\cQ^{j,1}_T)^2
}
	\end{pmatrix}.
\end{equation}
\end{enumerate}
Assume now that the process is ergodic (i.e., \eqref{eq:ergodic:cases} is satisfied). 
\begin{enumerate}[i)]
\setcounter{enumi}{1}
\item \label{th1:item:LLN:multi}
The following LLN holds, i.e.,~the estimator is (strongly) consistent:
$
\destim -\dvec
\convas[T\to\infty] 0.
$
\item \label{th1:item:CLT:multi}
The following CLT holds: 
\[\sqrt{T}
\left(\destimj -\dvecj\right)_{j=0}^d
\xrightarrow[{T\to \infty}]{\mathrm{stably}}  
\left({{\mathcal N}}_j\right)_{j=0}^d
\]
where
${{\mathcal N}}_j=\begin{pmatrix}
N^{j, \alpha}, & N^{j,\beta}
\end{pmatrix}$, 
$j=0,1\ldots,d$
are independent, independent of $X$, two-dimensional Gaussian random variables 
with covariance matrices respectively $\sigma_j^2 \Gamma_j^{-1}$ and $\sigma_j^2 \Gamma_j^{-1}$ 
where
\begin{equation} \label{th:erg:covariance:multi}
\Gamma_j :=  
\begin{pmatrix} \overline{\cQ}^{j,2}_\infty &  -\overline{\cQ}^{j,1}_\infty \\
 -\overline{\cQ}^{j,1}_\infty & \overline{\cQ}^{j,0}_\infty 
\end{pmatrix}, 
\end{equation}
and $\overline{\cQ}^{j,i}_\infty ,\, i\in \{0,1,2\}$
are real constants such that
$\lim_{t\rightarrow \infty} \frac{{\cQ}^{j,i}_t}{t} \overset{\as}= \overline{\cQ}^{j,i}_\infty$.
\item \label{th1:item:LAN:multi}
The LAN property holds for the likelihood evaluated at the true parameters 
$\dvec$
with rate of convergence
$\frac1{\sqrt{T}}$ : there exists a random vector ${A_T}$ such that for small perturbations
$\Delta \dvec := (\Delta a_j , \Delta b_j)_{j=0}^d$ it holds that
\[
\begin{split}
	& \log \frac{ G_T(\dvec + \frac{1}{\sqrt{T}}\Delta \dvec )}
	{G_T \dvec}
	 - \left( {A_T} \cdot  
	\Delta \dvec
	 - 
	\Delta \dvec \cdot
 \Gamma \Delta \dvec \right)
\end{split}
\]
converges to $0$ in probability as $T\to\infty$. The matrix
$\Gamma$ is the asymptotic Fisher information 
{\scriptsize \[ \Gamma 
	= \begin{pmatrix} \sigma_+^{-2} \Gamma_+ & 0_{\RR^{2\times 2}} & \cdots  & 0_{\RR^{2\times 2}} \\
 0_{\RR^{2\times 2}} &  \sigma_{d-1}^{-2} \Gamma_{d-1}   & \ddots & \vdots \\
 \vdots & \ddots & \ddots & 0_{\RR^{2\times 2}}
\\
 0_{\RR^{2\times 2}} & \cdots & 0_{\RR^{2\times 2}} & \sigma_-^{-2} \Gamma_- \end{pmatrix}.\]
}
\end{enumerate}
\end{theorem}

In order to study the asymptotic behavior of the estimator, we introduce a different expression for the estimators in~\eqref{eq:QMLE_ct:multi} based on the following notation.
Given $T\in (0,\infty)$, $j\in \{0,1,\ldots,d\}$, and $i\in \{0,1\}$ let 
\begin{equation*}
M_T^{j,i} :=
		\int_0^T (X_s)^i \ind{I_j}(X_s) \vd W_s.
\end{equation*}
Observe that~\eqref{eq:AffDOBM:multi} yields for $i\in \{0,1\}$, $j\in \{0,1\ldots,d\}$:
\begin{equation} \label{eq:equality:Ms:multi}
 	\cM_T^{j,i}
	= \sigma_j M^{j,i}_T 
	+ b_j \cQ^{j,i}_T 
	- a_j \cQ^{j,i+1}_T.
\end{equation}
Note that $\cQ^{\pm,0}_T \cQ^{\pm,2}_T - (\cQ^{\pm,1}_T)^2$ is $\PP$-a.s.~positive by Cauchy-Schwarz.
By~\eqref{eq:equality:Ms:multi}, we have the following reformulation of~\eqref{eq:QMLE_ct:multi}.

\begin{lemma}[Multi-threshold version of Lemma~\ref{lem:QMLE:ab}]\label{lem:QMLE:ab:multi}
Let $T\in (0,\infty)$ and $j\in\{0,1,\ldots,d\}$.
Equation~\eqref{eq:QMLE_ct:multi} can be expressed as
\begin{equation} \label{eq:MLE:asy:multi}	
	\destim_j 
	=\dvec_j 
	+ \sigma_j
	\begin{pmatrix}
		\tfrac{	M^{j,0}_T  \cQ^{j,1}_T -\cQ^{j,0}_T M^{j,1}_T}
	{ \cQ^{j,0}_T \cQ^{j,2}_T	-(\cQ^{j,1}_T)^2} ,
	&
	\tfrac{M^{j,0}_T \cQ^{j,2}_T-\cQ^{j,1}_T	M^{j,1}_T	}
	{\cQ^{j,0}_T\cQ^{j,2}_T-(\cQ^{j,1}_T)^2}
	\end{pmatrix},
\end{equation}
that can be rewritten as
\begin{equation} \label{eq:MLE:asy2:multi}	
\begin{pmatrix}
\alpha^{(j)}_T  \\ \beta^{(j)}_T
\end{pmatrix}
= \begin{pmatrix}
	a_j  \\  b_j
\end{pmatrix}
 + \sigma_j
 \begin{pmatrix}
	-1 & 0
	\\
	0 & 1
\end{pmatrix}
\begin{pmatrix}
	Q^{j,2}_T & Q^{j,1}_T
	\\
	Q^{j,1}_T & Q^{j,0}_T
\end{pmatrix}^{\! -1}
\begin{pmatrix}
	M_T^{j,1}  \\ M_T^{j,0}
\end{pmatrix}.
\end{equation}
\end{lemma}

The proofs of Lemma~\ref{lem:QMLE:ab:multi} and of Theorem~\ref{th:continuous:multi} are easily adapted from the case of a unique threshold ($d=1$) to the multi-threshold case and therefore omitted.

\subsection{Drift estimation from discrete observations}
\label{sec:est:disc}

We assume now to observe the process on the discrete time grid $0=t_0 < t_1 <\ldots<t_{N-1}< t_N=T$, 
 for $N\in \NN$, $T\in (0,\infty)$, and set $\Delta_N=\max_{k=1,\dots, N} \{t_{k}-t_{k-1}\}$. 
We define  $X_i\eqdef X_{t_i}$  with $i=0,\ldots,N$.

The discrete versions of \eqref{def:M:Q:multi} are defined as follows: for $m=0,1,2$, $j=0,\ldots,d$ let
\begin{equation} \label{def:M:Q:disc:multi}
\begin{split}
   & \cM^{j,m}_{T,N} := \sum_{k=0}^{N-1} 
     (X_{k+1}-X_k) X_k^m \ind{ I_j }(X_k),
\quad \text{and} \\
&\quad
    \cQ^{j,m}_{T,N} := \sum_{k=0}^{N-1} (t_{k+1}-t_k)
    X_k^m \ind{ I_j }(X_k).
\end{split}
\end{equation}
The \emph{discretized likelihood} 
$ G_{T,N} \dvec $
and the \emph{discretized quasi-likelihood}
$ \Lambda_{T,N} \dvec  $ are defined as in Section 2. 

For $N \in \NN$ and $T\in (0,\infty)$ let $\destimN{T}=\destimNj{T_N}_{j=0}^d$
with
\begin{equation} \label{eq:QMLE_dis_time:multi}	
\destimNj{T} 
=
\begin{pmatrix}
 	 \frac{\cM^{j,0}_{T,N} \cQ^{j,1}_{T,N} - \cQ^{j,0}_{T,N}
\cM^{j,1}_{T,N}	
	}{ 
\cQ^{j,0}_{T,N}
\cQ^{j,2}_{T,N}	-(\cQ^{j,1}_{T,N})^2},
		&
 	\frac{
\cM^{j,0}_{T,N} \cQ^{j,2}_{T,N}-\cQ^{j,1}_{T,N}	
\cM^{j,1}_{T,N}	}{
\cQ^{j,0}_{T,N}
\cQ^{j,2}_{T,N}-(\cQ^{j,1}_{T,N})^2
}
	\end{pmatrix}.
\end{equation}

\begin{theorem} \label{th:joint:CLT:multi}
{
Let $(T_N)_{N\in \NN}$ be a sequence in $(0,\infty)$.
For all $N\in \NN$ let 
$\destimN{T_N}$ 
be defined as in \eqref{eq:QMLE_dis_time:multi}.
\begin{enumerate}[i)]
\item \label{th2:item:QMLE:multi}
For every $N\in \NN$ the vector 
$\destimN{T_N}$
maximizes both the likelihood 
$G_{T_N,N}\dvec$ and the quasi-likelihood $\Lambda_{T_N,N}\dvec$.
\end{enumerate}
Assume that the process is ergodic and that $X_0$ follows {the stationary distribution} $\mu$. Moreover, assume
\[
\lim_{N\to\infty}T_N = \infty \quad \text{and} \quad \lim_{N\to\infty} \Delta_N =0 .
\]
\begin{enumerate}[i)]
\setcounter{enumi}{1}
\item \label{th2:item:LLN:multi}
The following LLN holds:
$ \destimN{T_N} \convp[N\to\infty] \dvec $
(the estimator is consistent). 
\item \label{th2:item:CLT:multi}
If $\lim_{N\to\infty}  T_N \Delta_N = 0$, the following CLT holds:
\[
\sqrt{T_N}
(\destimNj{T_N} - \dvecj)_{j=0}^d
\xrightarrow[N\to \infty]{\mathrm{stably}}
({\boldsymbol{\mathcal N}})_{j=0}^d
\]
where 
${\boldsymbol{\mathcal N}}$ is as in Theorem~\ref{th:continuous:multi}.
\item \label{th2:item:LAN:multi}
If $\lim_{N\to\infty} T_N \Delta_N = 0$, the analogous of the LAN property holds for the discretized likelihood evaluated at the true parameters 
with rate of convergence
$\frac1{\sqrt{T_N}}$ and asymptotic Fisher information 
$ \Gamma $ as in Theorem~\ref{th:continuous:multi}.
\end{enumerate}}
\end{theorem}

\subsection{Proof of Theorem~\ref{th:joint:CLT:multi}}


Analogously to the case $d=1$, the main ingredient of the proof is the following lemma.
\begin{lemma}[Multi-threshold version of Lemma~\ref{lemma:control:diff:disc:cont}]\label{lemma:control:diff:disc:cont:multi}
Assume the process is ergodic. 
Let $X$ be the solution to \eqref{eq:AffDOBM:multi}, with $X_0 $ distributed as the stationary distribution $\mu$, let $\lambda \in \{1,2\}$ be fixed,
and let $(T_N)_{N\in \NN} \subset (0,\infty)$ be a sequence satisfying, as $N \to\infty$, that $T_N\to \infty$ and 
$\lim_{N\to\infty} T_N^{1-1/\lambda} \sqrt{\Delta_N} = 0$ 
where $\Delta_N:=\sup_{k=1,\ldots,N} (t_{k}-t_{k-1})$.
Then for all $m\in \{0,1,2\}$, $j\in\{0,1\}$, $k\in\{0,\ldots,d\}$ it holds
\begin{equation*}
\limsup_{N\to\infty} {T_N^{-\nicefrac1{\lambda}}}
 \EE\left[ | \cQ^{k,m}_{T_N} - \cQ^{k,m}_{T_N,N} |\right] = 0
 \text{ and } 
\limsup_{N\to\infty} {T_N^{-\nicefrac1{\lambda}}} \EE\left[ |\cM^{k,j}_{T_N} - \cM^{k,j}_{T_N,N} | \right] = 0
\end{equation*}
where $\cQ^{k,m}_{T_N}$, $\cQ^{k,m}_{T_N,N}$,
$\cM^{k,j}_{T_N}$, $\cM^{k,j}_{T_N,N}$ are defined in~\eqref{def:M:Q:multi} and \eqref{def:M:Q:disc:multi}.
\end{lemma}

Before providing the proof of Lemma~\ref{lemma:control:diff:disc:cont:multi}, we show how it intervenes in the proof of Items~\eqref{th2:item:LLN:multi}-\eqref{th2:item:CLT:multi} of Theorem~\ref{th:joint:CLT:multi}. 
For all $N\in \NN$ 
\[
	\destimN{T_N} -\dvec = \destimN{T_N} - ({ \alpha_{T_N}, \beta_{T_N}}) + ({ \alpha_{T_N}, \beta_{T_N}})-\dvec
\]
The second term of the sum is handled with Theorem~\ref{th:continuous:multi} (more precisely Item~\eqref{th1:item:CLT:multi}) providing the desired limit distribution.
The first instead can be rewritten, using equations~\eqref{eq:QMLE_ct:multi} and~\eqref{eq:QMLE_dis_time:multi}, as an expression which involves only terms of the kind
\begin{footnotesize}
\[
\left(\frac{\cQ^{j,i}_{T_N,N}}{ \cQ^{j,0}_{T_N,N}\cQ^{j,2}_{T_N,N}	-(\cQ^{j,1}_{T_N,N})^2} 
- \frac{\cQ^{j,i}_{T_N}}{ \cQ^{j,0}_{T_N}\cQ^{j,2}_{T_N}	-(\cQ^{j,1}_{T_N})^2}  \right) \cM^{j,k}_{T_N}
+ \frac{\cQ^{j,i}_{T_N,N}  { ( \cM^{j,k}_{T_N,N} - \cM^{j,k}_{T_N} )} }{ \cQ^{j,0}_{T_N,N}\cQ^{j,2}_{T_N,N}	-(\cQ^{j,1}_{T_N,N})^2}
\] 
\end{footnotesize}
for $i\in \{0,1,2\}$, $j\in \{0,\ldots,d\}$, $k \in \{0,1\}$.

Combining Lemma~\ref{lemma:control:diff:disc:cont:multi} with Lemma~\ref{eq:Qinfinity2:multi} and 
Theorem~2.2 in \citep{crimaldi}  
ensures the consistency of the estimator if $\Delta_N \to 0$ as $N\to\infty$, and
if $T_N\Delta_N \to 0$ as $ N\to\infty$ then it
implies also that 
\[
	\sqrt{T_N} \left(\destimN{T_N} - ({ \alpha_{T_N}, \beta_{T_N}})\right) \convp[N\to\infty] 0.
\]

\begin{proof}[Proof of Lemma~\ref{lemma:control:diff:disc:cont:multi}]
The proof is similar to the one of Lemma~\ref{lemma:control:diff:disc:cont}. We provide here the main idea of the key step, that is the proof of the analogous of~\eqref{item:cond:1}: for all $j \in \{0,1\ldots,d\}$, $m\in \{0,1,2,3,4\}$
\begin{equation} \label{item:cond:1:multi}
	\int_0^{T_N} \EE\left[ |X_{\lfloor t \rfloor_{\Delta_N}}|^m \ind{\{X_{\lfloor t \rfloor_{\Delta_N}} \in I_j, X_t \not\in I_j\}} \right] \!\vd t 
\text{ is } o(T_N^{\nicefrac1\lambda}).
\end{equation}
We reduce to compute, given $X_{\lfloor t \rfloor_{\Delta_N}}$ for ${\lfloor t \rfloor_{\Delta_N}} = t_k$, the probability that the first exit time of a standard Brownian motion from a suitable symmetric interval is smaller than $t- t_{t_k}$:
$ p_{t} := \mathbb{P}( \tau_{X_{\lfloor t \rfloor_{\Delta_N}}, I_j} \leq t- \lfloor t \rfloor_{\Delta_N})$.
Indeed starting from $X_{t_k}\in I_j$ for some $k\in \{0,\ldots,N\}$, $j\in \{0,\ldots,d\}$ at a suitable distance $R_{X_{t_k}}$ from the boundary of $I_j$, if the Brownian motion driving the OU process does not exit in small time a suitable interval then the OU of parameters $(a_j,b_j,\sigma_j)$ stays in $(X_{t_k} -R_{X_{t_k}}, X_{t_k} +R_{X_{t_k}}) \subset I_j$ because the drift is small.
More precisely if $X_{t_k} \in I_j$ let $R_{X_{t_k}}:=\min\{X_{t_k}-r_j, r_{j+1}-X_{t_k}\}$, let $\tilde r_j \in \{r_j,r_{j+1}\}$ such that $R_{X_{t_k}}=|X_{t_k}-\tilde r_j|$, and let 
\[
	B_{X_{t_k}}:=(e^{-|a_j| (t_{k+1}-t_k)}-|a_j|(t_{k+1}-t_k)) R_{X_{t_k}} - \frac{ |b_j - a_j\tilde r_j |}{\sigma_j} (t_{k+1}-t_k).
\]
Note that $\sup_{j=0,\ldots,d}|\tilde r_j| <\infty$.
If $B_{X_{t_k}} >0$, then $\tau_{X_{t_k}, I_j}$ is the first exit time of a Brownian motion from the interval $[-B_{X_{t_k}}, B_{X_{t_k}}]$.
Moreover for $t\in (t_k,t_{k+1}]$ it holds that $t\mapsto p_t$ is increasing, and
\[
	p_{t} 
	\leq 2 \sum_{n=0}^\infty \text{erfc}\left(\tfrac{B_{X_{t_k}} (1+2n)}{\sqrt{t- {t_k}}}\right)
	\leq 2 \sum_{n=0}^\infty \exp{\!\left( -\tfrac{B_{X_{t_k}}^2 (1+2n)^2}{t- {t_k}}\right)}
\] 
and so $p_{t} \leq 2 \exp{\!\left( -\frac{B_{X_{t_k}}^2}{t- {t_k}}\right)} \left(1-\exp{\!\left( -8\frac{B_{X_{t_k}}^2}{t- {t_k}}\right)}\right)^{\!\!-1}$.

Then using that for every $t>0$, $N\in \NN$ it holds $1= \ind{\{B_{X_{\lfloor t \rfloor_{\Delta_N}}} > \sqrt{t- \lfloor t \rfloor_{\Delta_N}} \}} + \ind{\{B_{X_{\lfloor t \rfloor_{\Delta_N}}} \leq \sqrt{t- \lfloor t \rfloor_{\Delta_N}}\}}$,
we split the integrals into two parts to deal with in two different ways.

Note that
\[
\begin{split}
	& \int_0^{T_N} \EE\left[ |X_{\lfloor t \rfloor_{\Delta_N}}|^m \ind{\{X_{\lfloor t \rfloor_{\Delta_N}} \in I_j, X_t \not\in I_j\}} \ind{\{B_{X_{\lfloor t \rfloor_{\Delta_N}}} \leq  \sqrt{t- \lfloor t \rfloor_{\Delta_N}}\}} \right] \!\vd t 
	\\
	& \leq 
	 \sum_{k=0}^{N-1} (t_{k+1}-t_k) \EE\left[ |X_{t_k}|^m \ind{\{X_{t_k}\in I_j\}} \ind{\{B_{X_{t_k}} \leq  \sqrt{t_{k+1}- t_k}\}} \right] 
	\\
	& \leq
	2 \left(\sup_{x\in I_j} \mu(x) \right) \sum_{k=0}^{N-1} (t_{k+1}-t_k)^{3/2} \frac{1 + \frac{ |b_j - a_j\tilde r_j |}{\sigma_j} \sqrt{t_{k+1}-t_k}}{e^{-|a_j| (t_{k+1}-t_k)}-|a_j| (t_{k+1}-t_k) } 
	\\
	& \qquad \cdot \left(|\tilde r_j| + \frac{\sqrt{t_{k+1}-t_k} + \frac{ |b_j - a_j\tilde r_j |}{\sigma_j} (t_{k+1}-t_k) }{e^{-|a_j| (t_{k+1}-t_k)}-|a_j| (t_{k+1}-t_k) }\right)^m.
\end{split}
\]
Therefore  for $N$ big, since $\Delta_N$ is small, the latter quantity goes like $\sum_{k=0}^{N-1} (t_{k+1}-t_k)^{3/2} \leq T_N \sqrt{\Delta_N}$ 
and it is $o(T_N^{\nicefrac1\lambda})$.

The other integral satisfies:
\[
\begin{split}
	& \int_0^{T_N} \EE\left[ |X_{\lfloor t \rfloor_{\Delta_N}}|^m \ind{\{X_{\lfloor t \rfloor_{\Delta_N}} \in I_j, X_t \not\in I_j\}} \ind{\{B_{X_{\lfloor t \rfloor_{\Delta_N}}} > \sqrt{t- \lfloor t \rfloor_{\Delta_N}}\}} \right] \!\vd t 
	\\
	& \leq  
	\int_0^{T_N} \EE\left[ |X_{\lfloor t \rfloor_{\Delta_N}}|^m p_{X_{\lfloor t \rfloor_{\Delta_N}}}  \ind{\{X_{\lfloor t \rfloor_{\Delta_N}} \in I_j\}} \ind{\{B_{X_{\lfloor t \rfloor_{\Delta_N}}} > \sqrt{t- \lfloor t \rfloor_{\Delta_N}}\}} \right] \!\vd t
	\\ 
	& \leq \frac{2}{1-e^{-8}} 	\sum_{k=0}^{N-1} \int_{0}^{t_{k+1}-t_k} 
	\EE\left[ |X_{t_k}|^m e^{- \frac{B_{X_k}^2}{t} }  \ind{\{X_{t_k} \in I_j\}} \ind{\{B_{X_k} > \sqrt{t}\}} \right] \!\vd t.
\end{split}
\]
Let $g_{j,k}:=e^{-|a_j| (t_{k+1}-t_k)}-|a_j| (t_{k+1}-t_k) $, $f_{j,k}:= \frac{\frac{ |b_j - a_j\tilde r_j |}{\sigma_j} (t_{k+1}-t_k) }{g_{j,k}}$, and let $\mu$ denote the invariant measure, then
\[
\begin{split}
	& \EE\left[ |X_{t_k}|^m e^{- \frac{B_{X_k}^2}{t} }  \ind{\{X_{t_k} \in I_j\}} \ind{\{B_{X_k} > \sqrt{t}\}} \right] 
	\\
	& =
	\int_{r_j}^{r_{j+1}} |x|^m e^{- \frac{g_{j,k}^2 \left(|x-\tilde r_j|  - f_{j,k}\right)^2}{t} } \ind{\{|x-\tilde r_j| > f_{j,k} + \frac{\sqrt{t}}{g_{j,k}} \}} \mu(\!\vd x)
	\\
	& \leq  C_j \int_{r_j-\tilde r_j}^{r_{j+1}-\tilde r_j} ||x|+|\tilde r_j||^m e^{- \frac{g_{j,k}^2 \left(|x| - f_{j,k}\right)^2}{t} -\frac{a_j x^2 - 2 (b_j-a_j \tilde r_j)x}{\sigma_j^2}} \ind{\{|x| > f_{j,k} + \frac{\sqrt{t}}{g_{j,k}} \}}  \vd x
	\\
	& \leq  C_j \int_{r_j-\tilde r_j}^{r_{j+1}-\tilde r_j} ||x|+|\tilde r_j||^m e^{- \frac{g_{j,k}^2 \left(|x| - f_{j,k}\right)^2}{t} -\frac{a_j x^2 - 2 |b_j-a_j \tilde r_j||x|}{\sigma_j^2}} \ind{\{|x| > f_{j,k} + \frac{\sqrt{t}}{g_{j,k}} \}}  \vd x
	\\
	& \leq  C_j \int_{0}^{|I_j|} |x+|\tilde r_j||^m e^{- \frac{g_{j,k}^2 \left(x - f_{j,k}\right)^2}{t} -\frac{a_j x^2 - 2 |b_j-a_j \tilde r_j|x}{\sigma_j^2}} \ind{\{x > f_{j,k} + \frac{\sqrt{t}}{g_{j,k}} \}}  \vd x
\end{split}
\]
where the constant $C_j$ may change from line to line and $|I_j|$ is the length of the interval.
For $N$ big enough $\Delta_N <<1$, $f_{j,k} + \frac{t_{k+1}-t_k}{g_{j,k}} < |I_j|$, $g_{j,k} \in [1/2,1] $, $f_{j,k}\leq 2|b_j-a_j \tilde r_j|/\sigma_j$ hence the latter integral is bounded from above by
\[
\begin{split}
	& C_j \int_{f_{j,k} + \frac{\sqrt{t}}{g_{j,k}}}^{|I_j|} |x+|\tilde r_j||^m e^{- \frac{g_{j,k}^2 \left(x - f_{j,k}\right)^2}{t} -\frac{a_j x^2 - 2 |b_j-a_j \tilde r_j|x}{\sigma_j^2}}  \vd x
	\\
	& = C_j \int_{\frac{\sqrt{t}}{g_{j,k}}}^{|I_j|} |x+|\tilde r_j||^m e^{- \frac{g_{j,k}^2 x^2}{t} -\frac{a_j (x+f_{j,k})^2 - 2 |b_j-a_j \tilde r_j|(x+f_{j,k})}{\sigma_j^2}}  \vd x
	\\
	& \leq C_j  \sqrt{t} \int_{0}^{+\infty} |x+|\tilde r_j||^m e^{- x^2 +\frac{8 x}{\sigma_j}}  \vd x
\end{split}
\]
where the constant $C_j$ may change from line to line.
Thus, for some positive constant $C_j$ it holds that
\[
\begin{split}
	& \int_0^{T_N} \EE\left[ |X_{\lfloor t \rfloor_{\Delta_N}}|^m \ind{\{X_{\lfloor t \rfloor_{\Delta_N}} \in I_j, X_t \not\in I_j\}} \ind{\{B_{X_{\lfloor t \rfloor_{\Delta_N}}} > \sqrt{t- \lfloor t \rfloor_{\Delta_N}}\}} \right] \!\vd t 
	\\ 
	& \leq C_j 	\sum_{k=0}^{N-1} \int_{0}^{t_{k+1}-t_k} \sqrt{t} \vd t = \frac{2}{3} C_j \sum_{k=0}^{N-1} (t_{k+1}-t_k)^{3/2} \leq \frac23 C_j T_N \sqrt{\Delta_N}
\end{split}
\]
and it is $o(T_N^{\nicefrac1\lambda})$.
The proof is thus completed.
\end{proof}

\addcontentsline{toc}{section}{Bibliography}
\bibliographystyle{abbrvnat}
\setlength{\bibsep}{0pt plus 0.5ex}
\begin{small}
\bibliography{biblio}
\end{small}

\end{document}